\documentclass[12pt]{article}
\input{epsf.sty}
\usepackage{latexsym,amsmath,amsfonts,amscd}
\usepackage{graphicx,epsfig,epstopdf}
\usepackage{changebar}
\usepackage{pst-plot}
\usepackage{multirow}
\usepackage{subfigure}
\usepackage{placeins}
\usepackage{amssymb}
\usepackage{bbm}

\usepackage[utf8]{inputenc}

\usepackage{amsmath}
\usepackage{amsthm}
\usepackage{amssymb}
\usepackage{authblk}
\usepackage[titletoc,toc,title]{appendix}

\newcommand{\R}{\mathbb{R}}
\newcommand{\pd}{\partial}
\newcommand{\beq}{\begin{equation}}
\newcommand{\eeq}{\end{equation}}

\newtheorem{thm}{Theorem}[section]
\newtheorem{lmm}[thm]{Lemma}
\newtheorem{rem}[thm]{Remark}
\newtheorem{defini}[thm]{Definition}

\renewcommand{\thefigure}{\thesection.\arabic{figure}}
\renewcommand{\thetable}{\thesection.\arabic{table}}

\allowdisplaybreaks
\DeclareMathOperator*{\argmin}{arg\,min}
\DeclareMathOperator*{\sgn}{sgn}

\title
{
Eventual linear convergence of the Douglas-Rachford iteration for basis pursuit
}

\vspace{.1in}
\author[]{Laurent Demanet\thanks{laurent@math.mit.edu} }

\author[]{Xiangxiong Zhang\thanks{zhangxx@math.mit.edu}}

\affil[]{Massachusetts Institute of Technology, Department of Mathematics, \\
 77 Massachusetts Avenue, Cambridge, MA 02139}

\begin{document}
\date{December 2012, revised May 2013 }
\maketitle

\vspace*{.10in}


\vspace{.3in}

\centerline{\bf Abstract}

\vspace{0.1in}

We provide a simple analysis of the Douglas-Rachford splitting algorithm in the
context of $\ell^1$ minimization with linear constraints, and quantify the
asymptotic linear convergence rate in terms of
principal angles between relevant vector spaces. 
In the compressed sensing setting, we show how to bound this rate in terms of the restricted isometry
constant. More general iterative schemes obtained by $\ell^2$-regularization
and over-relaxation including the dual split Bregman method \cite{LBSB} are also treated,
which answers the question how to choose the relaxation and soft-thresholding parameters to accelerate the asymptotic convergence rate.
We make no attempt at characterizing the
transient regime preceding the onset of linear convergence.

\vfill

\noindent {\bf Acknowledgments:}
The authors are grateful to Jalal Fadili, Stanley Osher, Gabriel Peyr\'{e}, Ming Yan, Yi Yang and Wotao Yin for discussions on modern methods of optimization that were very instructive to us. The authors are supported by the National Science Foundation and the Alfred P. Sloan Foundation.

\bigskip


\section{Introduction}
\setcounter{equation}{0}
\setcounter{figure}{0}
\setcounter{table}{0}

\subsection{Setup}

In this paper we consider certain splitting algorithms for basis pursuit \cite{Chen98atomicdecomposition}, the constrained optimization problem
\beq\label{BP}
\min \| x \|_1 \qquad \mbox{s.t.} \qquad Ax=b.
\eeq
Throughout this paper we consider $A \in \R^{m \times n}$ with $m \leq n$, and we assume that $A$ has full row rank. We also assume that the solution $x^*$ of (\ref{BP}) is unique.

In particular, we treat splitting algorithms that naturally arise in the scope of minimization problems of the form
\[
\min_x f(x) + g(x),
\]
where $f$ and $g$ are convex, lower semi-continuous (but not otherwise smooth), and have simple resolvents of their subdifferentials
\[
J_{\gamma F}=(\mathrm{I}+\gamma F)^{-1}, \qquad J_{\gamma G}=(\mathrm{I}+\gamma G)^{-1},
\]
where $F = \pd f(x)$ and $G = \pd g(x)$ are the respective subdifferentials of $f$ and $g$ at $x$. In those terms, $x$ is a minimizer if and only if $0 \in F(x) + G(x)$. Resolvents are also often called proximal operators, as they obey $J_{\gamma F}(x)=\argmin_z \gamma f(z) + \frac12\|z-x\|^2$. In the case of basis pursuit, it is well known that
\begin{itemize}
\item $f(x) = \| x \|_1$ and $g(x) = \iota_{\{x:Ax=b\}}$, the indicator function equal to zero when $Ax=b$ and $+\infty$ otherwise;
\item $J_{\gamma F}$ is soft-thresholding (shrinkage) by an amount $\gamma$,
\[
J_{\gamma F}(x)_i = S_\gamma(x)_i = \sgn(x_i)\max \{|x_i|-\gamma,0\};
\]
\item $J_{\gamma G}$ is projection onto the set $Ax=b$, namely
\[
J_{\gamma G}(x) =\mathrm P(x)=x+A^+(b-Ax),
\]
with $A^+=A^T(AA^T)^{-1}$ denoting the pseudo inverse.
\end{itemize}
The simplest splitting algorithm based on the resolvents is
\[
x^{k+1} = J_{\gamma F} J_{\gamma G} x^k.
\]
This iteration is successful in the special case when $f$ and $g$ are both indicators of convex sets, but does not otherwise generally enjoy good convergence properties. Instead, one is led to consider reflection operators $R_{\gamma F} = 2 J_{\gamma F} - I$, $R_{\gamma G} = 2 J_{\gamma G} - I$, and write the \emph{Douglas-Rachford splitting} \cite{lions, Eckstein92onthe}
\beq\label{alg1} \begin{cases}
y^{k+1} = \frac12 ( R_{\gamma F} R_{\gamma G} + I ) y^k=J_{\gamma F}\circ(2J_{\gamma G}-I)y^k+(I-J_{\gamma G})y^k, \\
x^{k+1} = J_{\gamma G} y^{k+1},
\end{cases}\eeq
where $I$ is the identity. The operator $T_\gamma = \frac12 ( R_{\gamma F} R_{\gamma G} + I )$ is \textit{firmly non-expansive} 
regardless of $\gamma > 0$
\cite{lions}. Thus ${y}^k$ converges to one of its fixed points ${y}^*$. Moreover, $x^*=J_{\gamma G}({y}^*)$ is one solution to $0\in F({x})+G({x})$.

For general convex functions $f(x)$ and $g(x)$, the sublinear convergence rate  $\mathcal{O}(1/k)$  of the algorithm (\ref{alg1}) was
proven for averages of iterates in \cite{Chambolle:2011:FPA:1968993.1969036, doi:10.1137/110836936}. 
The firm non-expansiveness also implies $\|y^k-y^{k-1}\|^2\leq \frac{1}{k}\|y^{0}-y^*\|^2$, see Appendix \ref{appendix1}.
Convergence questions for the Douglas-Rachford splitting were recently studied in the context of projections onto possibly nonconvex sets \cite{Borwein, HesseLuke} with potential
applications to phase retrieval \cite{Bauschke:02}.

In the case of basis pursuit, we note that the Douglas-Rachford (DR) iteration takes the form
\begin{equation}\label{alg2}
\begin{cases} {y}^{k+1}=  S_\gamma(2 x^k -y^k) +y^k -x^k,\\
{x}^{k+1}={y}^{k+1}+A^+(b-Ay^{k+1})\end{cases}.
\end{equation}

\subsection{Main result}

In practice, (\ref{alg2}) often settles into a regime of linear convergence. See Figure \ref{ex1} for an illustration of a typical error curve where the matrix ${A}$ is a $3\times 40$ random matrix and $x^*$ has three
nonzero components. Notice that the error $\|y^k-y^*\|$ is monotonically decreasing since the operator $T_\gamma$ is non-expansive. The same cannot be said of $\|x^k-x^*\|$.

In this example, the regime of linear convergence was reached quickly for the $y^k$. That may not in general be the case, particularly if $AA^T$ is ill-conditioned. Below, we provide the characterization of the error decay rate in the linear regime. To express the result, we need the following notations.

Assume that the unique solution $x^*$ of (\ref{BP}) has $r$ zero components.   Let $e_i$ ($i=1,\cdots,n$) be the standard basis in $\mathbbm{R}^n$. Denote the basis vectors corresponding to zero components in
${x}^*$ as $e_j$ ($j=i_1,\cdots,i_r$). Let ${B}$ be the $r \times n$ selector of
the zero components of ${x}^*$, i.e., $B =[e_{i_1},\cdots,e_{i_r}]^T$.  Let $\mathcal{N}({A})=\{{x}:{Ax}=0\}$ denote the nullspace of $A$ and $\mathcal{R}({A^T})=\{{x}:x={A^T z}, z\in \mathbbm{R}^m\}$ denote the range of $A^T$.

Then, for the numerical example discussed earlier, the slope of  $\log \|y^k-y^*\|$ as a function of $k$ is $\log\left(\cos\theta_1\right)$ for large $k$, where $\theta_1$ is the first principal angle between $\mathcal{N}(A)$ and $\mathcal{N}(B)$.
See Definition \ref{def1} in Section \ref{subsection23} for principal angles between subspaces.

\begin{figure}[ht]
  \begin{center}
\includegraphics[scale=0.5]{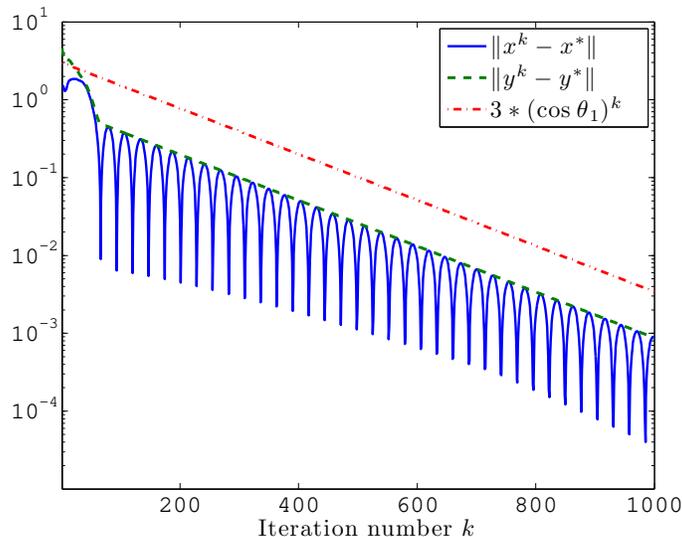}
  \end{center}
  \caption{A typical error curve for Douglas-Rachford}
  \label{ex1}
\end{figure}

Our main result is that the rate of decay of the error is indeed $\cos \theta_1$ for a large class of situations that we call \emph{standard}, in the sense of the following definition.

\begin{defini}
Consider a basis pursuit problem $(b,A)$ with solution $x^*$. Consider $y^0$ an initial value for the Douglas-Rachford iteration, and $y^* = \lim_{k \to \infty} T_\gamma^k y^0$.

Consider the preimage of the soft thresholding of all vectors with the same signs as $x^*$:
$$\mathcal{Q}=\{S_\gamma^{-1}(x): \sgn(x)=\sgn(x^*)\}
=Q_1\otimes Q_2\otimes\cdots Q_n,$$
where $$Q_j=\begin{cases}
            (\gamma,+\infty), & \mbox{if } x^*_j>0 \\
            (-\infty,-\gamma), & \mbox{if } x^*_j<0 \\
            [-\gamma,\gamma], & \mbox{otherwise}
           \end{cases}.$$
We call $(b,A; y^0)$ a \emph{standard} problem for the Douglas-Rachford iteration if $\mathrm R(y^*)$ belongs to the interior of $\mathcal{Q}$, where $\mathrm R$ is the reflection operator defined earlier. In that case, we also say that the fixed point $y^*$ of $T_\gamma$ is an \emph{interior fixed point}. Otherwise, we say that $(b,A; y^0)$ is nonstandard for the Douglas-Rachford iteration, and that $y^*$  is a \emph{boundary fixed point}.
\end{defini}

\begin{thm}
\label{thm1}
Consider $(b,A; y^0)$ a standard problem for the Douglas-Rachford iteration, in the sense of the previous definition. Then the Douglas-Rachford iterates $y^k$ obey
\[
\| y^k - y^* \| \leq C \left( \cos \theta_1 \right)^k,
\]
where $C$ may depend on $b$, $A$ and $y^0$ (but not on $k$), and $\theta_1$ is
the leading principal angle between $\mathcal{N}(A)$ and $\mathcal{N}(B)$.
\end{thm}

The auxiliary variable $y^k$ in (\ref{alg2}) converges linearly for
sufficiently large $k$, thus  $x^k$ is also bounded by a linearly
convergent sequence since $\|x^k-x^*\|=\|\mathrm P(y^k)-\mathrm P(y^*)\|=
\|\mathrm P(y^k-y^*)\|\leq\|y^k-y^*\|$.

Intuitively, convergence enters the linear regime when the support of the iterates essentially matches that of $x^*$. By essentially, we mean that there is some technical consideration (embodied in our definition of a ``standard problem") that this match of supports is not a fluke and will continue to hold for all iterates from $k$ and on. When this linear regime is reached, our analysis in the standard case hinges on the simple fact that $T_\gamma(y^k)-y^*$ is a linear transformation on $y^k-y^*$ with an eigenvalue of maximal modulus equal to $\cos \theta_1$.

In the nonstandard case ($y^*$ being a boundary fixed point), we furthermore show that the rate of convergence for $y^k$ is \emph{generically} of the form $\cos \bar  \theta_1$, where $0<\bar  \theta_1 \leq \theta_1$ is the leading principal angle between $\mathcal{N}(A)$ and $\mathcal{N}(\bar B)$, with $\bar B$ a submatrix of $B$ depending on $y^*$.
Nongeneric cases are not a priori excluded by our analysis, but have not been observed in our numerical tests. See Section \ref{subsection25} for a discussion of the different types of nonstandard cases.

\subsection{Regularized basis pursuit}

In practice, if $\theta_1$ is very close to zero, linear convergence with rate $ \cos \theta_1 $ might be very slow.  The following regularized problem is often used to accelerate convergence,
\begin{equation}
\label{BP2} \quad \min_{{x}}\left\{\|{x}\|_1+\frac{1}{2\alpha}\|x\|^2: Ax=b\right\}.
\end{equation}
It is proven in \cite{Yin:2010:AGL:2078411.2078424} that there exists a
$\alpha_\infty$
such that the solution of (\ref{BP2}) with $\alpha\geq \alpha_\infty$ is the
solution
of (\ref{BP}). See \cite{yin2} for more discussion of $\alpha_\infty$. For the rest of this paper, we assume $\alpha$ is taken large enough so that
$\alpha\geq \alpha_\infty$.

For all the discussion regarding $\ell^2$-regularized basis pursuit, it is convenient to make the technical assumption that  $\theta_1\leq \frac{\pi}{4}$. Notice that regularization is probably unwarranted in the event $\theta_1 > \pi /4$, since $\cos\theta_1$ would be a very decent linear convergence rate.


In particular, the Douglas-Rachford splitting (\ref{alg1}) with $f(x)=\|x\|_1+\frac{1}{2\alpha}\|x\|^2$ and $g(x)=\iota_{\{x:Ax=b\}}$ is equivalent to the dual split Bregman method for basis pursuit \cite{LBSB}, which will be discussed in Section \ref{seclbsb}.

\begin{figure}[ht]
  \begin{center}
\subfigure{\includegraphics[scale=0.6]{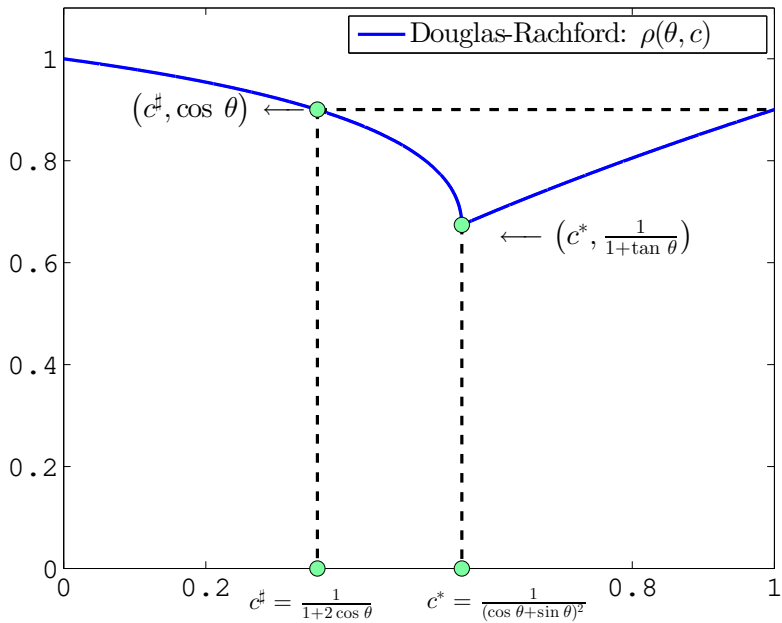}}
  \end{center}
  \caption{An illustration of the rate of linear convergence for Douglas-Rachford splitting on $\ell^2$-regularized basis pursuit: $\rho(\theta,c)$ for a fixed $\theta$. The vertical axis is $\rho(\theta,c)$ and the horizontal axis is $c=\frac{\alpha}{\alpha+\gamma}$. The case $\alpha = +\infty$ (unregularized DR) is at $c = 1, \rho = \cos \theta$.}
  \label{bestc}
\end{figure}

With the same assumptions and notations as in Theorem \ref{thm1}, assuming $\theta_1\leq\frac{\pi}{4}$, for the Douglas-Rachford splitting (\ref{alg1}) on (\ref{BP2}),  we have
$\| y^k - y^* \| \leq C \rho(\theta_1,c) ^k$ where $c=\frac{\alpha}{\alpha+\gamma}$ and
\[
 \rho(\theta,c)=\begin{cases}
\sqrt{c}\cos\theta, & \mbox{if } c\geq\frac{1}{(\cos\theta+\sin\theta)^2}\\
\frac12\left(c\cos(2\theta)+1+\sqrt{\cos^2(2\theta)c^2-2c+1} \right)& \mbox{if }
c\leq \frac{1}{(\cos\theta+\sin\theta)^2}
                   \end{cases}.
\]

Let $c^*=\frac{1}{(\cos\theta_1+\sin\theta_1)^2}$ which is equal to
$\argmin_c\rho(\theta_1,c)$. Let $c^\sharp=\frac{1}{1+2\cos\theta_1}$
which is the solution to $\rho(\theta_1,c)=\cos\theta_1$. See Figure
\ref{bestc}. Then for any $c\in (c^\sharp,1)$, we have
$\rho(\theta_1,c)<\cos\theta_1
$. The asymptotic convergence rate of (\ref{alg1}) on (\ref{BP2}) is faster than
(\ref{alg2}) if $\frac{\alpha}{\alpha+\gamma}\in (c^\sharp,1)$.
The best achievable asymptotic convergence rate is
$\rho(\theta_1,c^*)=\sqrt{c^*}\cos\theta_1=\frac{\cos\theta_1}{
\cos\theta_1+\sin\theta_1} = \frac{1}{1 + \tan \theta_1}$ when $\frac{\alpha}{\alpha+\gamma}=c^*$.

\subsection{Generalized Douglas-Rachford and Peaceman-Rachford}

The generalized Douglas-Rachford splitting introduced in \cite{Eckstein92onthe} can be written as
\beq
\label{GDR}
y^{k+1}=(1-\lambda_k) y^k+\lambda_k\frac{R_{\gamma F}R_{\gamma  G}+I}{2}y^k,\quad \lambda_k\in(0,2),
\eeq
We have the usual DR splitting when $\lambda_k = 1$. In the limiting case $\lambda_k = 2$, (\ref{GDR}) becomes the Peaceman-Rachford (PR) splitting \beq
\label{PR}y^{k+1}=R_{\gamma  F}R_{\gamma  G}y^k.\eeq

Consider (\ref{GDR}) with constant relaxation parameter $\lambda\in(0,2]$ on (\ref{BP2}). 
With the same assumptions and notations as in Theorem \ref{thm1}, assuming $\theta_1\leq\frac{\pi}{4}$,  we have
the eventual linear convergence rate
$\| y^k - y^* \| \leq C \rho(\theta_1,c,\lambda) ^k$ where $c=\frac{\alpha}{\alpha+\gamma}$ and
\[
 \rho(\theta,c,\lambda)=\begin{cases}
\sqrt{c\sin^2\theta\lambda^2-(1-c\cos(2\theta))\lambda+1}, & \mbox{if } c\geq\frac{1}{(\cos\theta+\sin\theta)^2}\\
\frac12\left(\lambda c\cos(2\theta)-\lambda+2+\lambda\sqrt{\cos^2(2\theta)c^2-2c+1} \right)& \mbox{if }
c\leq \frac{1}{(\cos\theta+\sin\theta)^2}
                   \end{cases}.
\]

For fixed $\theta$ and $c$, the optimal relaxation parameter is
\[\lambda^*(\theta,c)=\argmin\limits_\lambda\rho(\theta,c,\lambda)=\begin{cases}2  & \mbox{if } c\leq
\bar c= \frac{1}{2-\cos(2\theta)} \\  \frac{\frac{1}{c}-\cos{2\theta}}{1-\cos{(2\theta)}} & \mbox{if } c\geq \bar c \end{cases},\]
which is a continuous non-increasing function with respect to $c$ and has range $(1,2]$ for $c\in(0,1)$.


The convergence rate at the optimal $\lambda = \lambda^*$ is
\[\rho(\theta,c,\lambda^*)=\begin{cases}c\cos(2\theta)+\sqrt{\cos^2(2\theta)c^2-2c+1}, & \mbox{if } c\leq
c^*=\frac{1}{(\cos\theta+\sin\theta)^2}\\ 
\sqrt{2c-1}, & \mbox{if } c^*\leq c\leq \bar c= \frac{1}{2-\cos(2\theta)}\\
\frac{\sqrt{2c-1-c^2\cos^2{(2\theta)}}}{2\sin\theta\sqrt{c}},
& \mbox{if } c\geq \bar c
\end{cases}.\]

See Figure \ref{bestc2} for the illustration of the asymptotic linear rate $\rho(\theta,c,\lambda)$.
Several interesting facts can be seen immediately:

\begin{enumerate}
\item For Peaceman-Rachford splitting, i.e., (\ref{GDR}) with $\lambda=2$, if $c\geq c^*$, the asymptotic rate $\rho(\theta,c,2)=\sqrt{2c-1}$ is independent of $\theta$.
\item For any  $c< \tilde c=\frac{1}{2-\cos^2\theta}$, we have $\rho(\theta,c,2)
<\rho(\theta,c,1)$, i.e., the Peaceman-Rachford splitting has a better convergence rate than Douglas-Rachford.
\item The best possible rate of (\ref{GDR}) is $\min\limits_{c,\lambda}\rho(\theta,c,\lambda)=\rho(\theta,c^*,2)=\frac{1-\tan\theta}{1+\tan\theta}$.
\end{enumerate}

\begin{figure}[ht]
  \begin{center}
\subfigure{\includegraphics[scale=0.6]{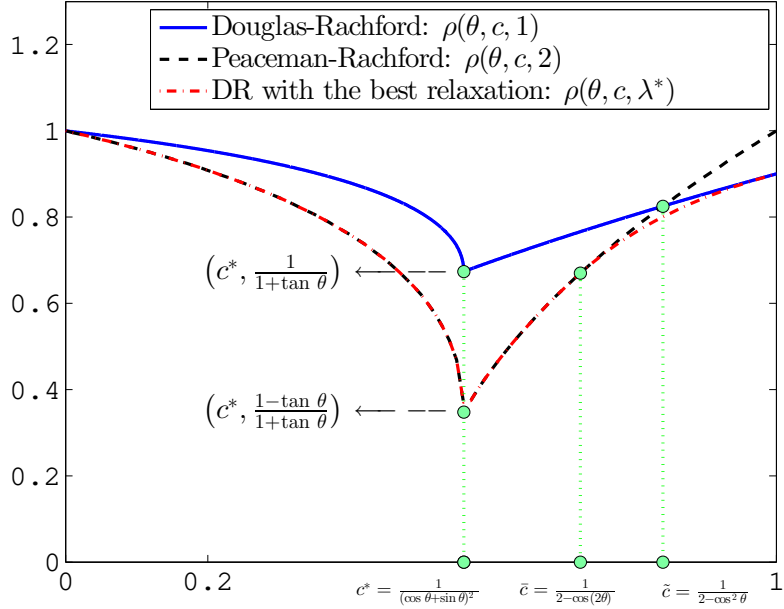}}
  \end{center}
  \caption{An illustration of eventual linear convergence rate for generalized Douglas-Rachford splitting with constant relaxation parameter $\lambda$ on $\ell^2$-regularized basis pursuit: $\rho(\theta,c,\lambda)$ for a fixed $\theta$. The vertical axis is $\rho(\theta,c,\lambda)$ and the horizontal axis is $c=\frac{\alpha}{\alpha+\gamma}$. For $c\leq \bar c$,
the best relaxation parameter is $\lambda^*=2$. }
  \label{bestc2}
\end{figure}

\subsection{Context}

There is neither strong convexity nor Lipschitz continuity in the objective function of (\ref{BP})
even locally around $x^*$, 
but any $x^k$ with the same support as $x^*$
lies on a low-dimensional manifold, on which the objective function $\|x\|_1$ is smooth. Such property is characterized as \textit{partial smoothness} \cite{Lewis:2002:ASN:588908.589340}.  In other words, it is not surprising that nonsmooth optimization algorithms for (\ref{BP})
converge linearly if $x^k$ has the correct support. For example, see \cite{hale2008fixed, Yin2013}.

The main contribution of this paper is the quantification of the asymptotic linear convergence rate 
for Douglas-Rachford splitting on basis pursuit.
It is well-known that Douglas-Rachford on the dual problem is the same as the alternating direction method of multipliers (ADMM) \cite{gabay83}, which is also equivalent to split Bregman method \cite{Goldstein:2009:SBM:1658384.1658386}. Thus the analysis in this paper also applies to ADMM on the dual problem of $\ell^2$-regularized basis pursuit, i.e., the dual split Bregman method for basis pursuit \cite{LBSB}.
By analyzing the generalized Douglas-Rachford introduced in \cite{Eckstein92onthe} including the Peaceman-Rachford splitting, we obtain
the explicit dependence of the asymptotic convergence rate on the parameters.

\subsection{Contents}
Details and proof of the main result will be shown in Section 2. In Sections 3, we apply the same methodology to obtain the asymptotic convergence rates for Douglas-Rachford, generalized Douglas-Rachford and Peaceman-Rachford splittings on the $\ell^2$-regularized basis pursuit.
In Section 4, we discuss the equivalence between Douglas-Rachford and dual split Bregman method, and their practical relevance.  Numerical experiments illustrating the theorems are shown. 
\section{Douglas-Rachford for Basis Pursuit}
\setcounter{equation}{0}
\setcounter{figure}{0}
\setcounter{table}{0}
\label{sec2}

\subsection{Preliminaries}

For any subspace $\mathcal{X}$ in $\mathbbm{R}^n$, we use
$\mathbbm{P}_\mathcal{X}(z)$ to
denote the orthogonal projection onto $\mathcal{X}$ of the point $z\in \mathbbm{R}^n.$

In this section, we denote $F(x)= \partial \|{x}\|_1$, $G(x)=\partial
\iota_{\{x:Ax=b\}}$,
 and
the resolvents are $J_{\gamma F}(x)=S_\gamma(x)$ and $J_{\gamma
G}(x)=\mathrm P(x)=x+A^+(b-Ax)$.
For convenience, we use $\mathrm R=2 \mathrm P -\mathrm I$ to denote reflection about $Ax=b$, i.e.,
$\mathrm R(x)=x+2A^+(b-Ax).$ It is easy to see that $\mathrm R$ is idempotent. Then $T_\gamma=S_{\gamma} \circ \mathrm R +\mathrm{I}-\mathrm P$.

Let $N(x^*)$ denote the set of coordinate indices associated with the nonzero
components of $x^*$,
namely, $N(x^*)\cup \{i_1,\cdots, i_r\}=\{1,\cdots,n\}$. Recall the definition
of $\mathcal{Q}$ in the previous section. Then for any $z\in \mathcal{Q}$,
the soft thresholding operator can be written as
$S_\gamma(z)=z-\gamma\sum\limits_{j\in N(x^*)}\sgn(x^*_j)e_j-B^+B z$.

\begin{lmm}
 \label{lmm0}
The assumption that $x^*$ is the unique minimizer of (\ref{BP}) implies
 $\mathcal{N}(A)\cap \mathcal{N
}(B)=\{\mathbf{0}\}$.
\end{lmm}
\begin{proof}
Suppose there exists a nonzero vector $z\in\mathcal{N}(A)\cap \mathcal{N
}(B)$. 
For any $\varepsilon\in \mathbbm{R}$ with small magnitude, 
we have $\sgn(x^*+\varepsilon z)^T=\sgn(x^*)^T$ and $A(x^*+\varepsilon z)=b$.
For nonzero small $\varepsilon$, the uniqueness of the minimizer implies 
$\|x^*\|_1<\|x^*+\varepsilon z\|_1=\sgn(x^*+\varepsilon z)^T(x^*+\varepsilon z)=
\sgn(x^*)^T(x^*+\varepsilon z)=\|x^*\|_1+\varepsilon \sgn(x^*)^Tz$. 
Thus $\sgn(x^*)^Tz\neq 0$.

 On the other hand,
for the function $h(\varepsilon)=\|x^*+\varepsilon z\|_1=\|x^*\|_1+\varepsilon \sgn(x^*)^Tz$ on a small neighborhood of $\varepsilon=0$,
the minimum of $h(\varepsilon)$ is $h(0)$, thus $\sgn(x^*)^Tz=h'(0)=0$.
This contradicts with the fact that $\sgn(x^*)^Tz\neq 0$.
\end{proof}

The sum of the dimensions of $\mathcal{N}(A)$ and $\mathcal{N}(B)$ should be
no larger than $n$
since $\mathcal{N}(A)\cap \mathcal{N
}(B)=\{\mathbf{0}\}$. Thus, $n-m+n-r\leq n$ implies $m\geq n-r.$

$\mathcal{N}(A)\cap\mathcal{N}(B)=\{\mathbf{0}\}$ also implies the orthogonal
complement of
the subspace spanned by $\mathcal{N}(A)$ and $\mathcal{N}(B)$ is
$\mathcal{R}(A^T)\cap \mathcal{R}(B^T)$.
Therefore, the dimension of $\mathcal{R}(A^T)\cap \mathcal{R}(B^T)$ is $m+r-n$.

\subsection{Characterization of the fixed points of $T_\gamma$}
Since $\partial \iota_{\{x:{Ax}={b}\}}=\mathcal{R}(A^T)$, the first order
optimality condition for (\ref{BP}) reads $0\in
\partial\|x^*\|_1+\mathcal{R}(A^T)$, thus
$\partial\|x^*\|_1\cap\mathcal{R}(A^T)\neq\varnothing$.
Any such $\eta\in
\partial\|x^*\|_1\cap\mathcal{R}(A^T)$ is called a dual certificate.

We have the following characterization of the fixed points of $T_\gamma$.

\begin{lmm}
\label{lemma1}
 The set of the fixed points of  $T_\gamma$ can be described as
$$\{y^*: y^*=x^*-\gamma \eta, \eta\in
\partial\|x^*\|_1\cap\mathcal{R}(A^T)\}.$$
Moreover, for any two fixed
points $y^*_1$ and $y^*_2$, we have $y^*_1-y^*_2, \mathrm R(y^*_1)-\mathrm
R(y^*_2)\in \mathcal{R}({A^T})\cap\mathcal{R}({B^T})$.
Thus
there is a unique fixed point $y^*$ if and only if $\mathcal{R}(A^T)\cap
\mathcal{R}(B^T)=\{\mathbf{0}\}$.
\end{lmm}
\begin{proof}
 For any $\eta\in
\partial\|x^*\|_1\cap\mathcal{R}(A^T)$, consider the vector $y^*=x^*-\gamma
\eta.$
Since $Ax^*=b$ and $A^+A\eta=\eta$ (implied by $\eta\in \mathcal{R}(A^T)$),
we have $\mathrm P(y^*)=y^*+A^+(b-Ay^*)=x^*-\gamma \eta+A^+(b-Ax^*+A \gamma
\eta)=x^*+A^+(b-Ax^*)=x^*$.
Further, $\eta\in
\partial\|x^*\|_1$ implies $S_\gamma(x^*+\gamma \eta)=x^*.$ Thus
$T_\gamma(y^*)=S_\gamma(2x^*-y^*)+y^*-x^*=S_\gamma(x^*+\gamma
\eta)-x^*+y^*=y^*.$

Second, for any fixed point $y^*$ of the operator $T_\gamma$, let
$\eta=(x^*-y^*)/\gamma$.
Then
\begin{equation}
\label{eq1}
 \mathrm  P(y^*)=x^*, \quad \mbox{(see Theorem 5 in \cite{Eckstein92onthe})}
\end{equation}
 implies $\eta=A^+A\eta$, thus $\eta\in\mathcal{R}(A^T)$.
Further, $y^*=T_\gamma(y^*)$ implies $S_\gamma(x^*+\gamma \eta)=x^*.$
We have $x^*=\argmin_z \gamma\|z\|_1+\frac12\|z-(x^*+\gamma \eta)\|^2$, thus
$\eta\in\partial\|x^*\|_1.$

Finally, let $y^*_1$ and $y^*_2$ be two fixed
points. Then $y^*_1-y^*_2= -\gamma (\eta_1-\eta_2)$ and
$\mathrm R(y^*_1)-\mathrm R(y^*_2)= \gamma (\eta_1-\eta_2)$
for some $\eta_1,\eta_2\in \partial\|x^*\|_1\cap\mathcal{R}(A^T)$.
Notice that $\eta_1,\eta_2\in \partial\|x^*\|_1$ implies $\eta_1-\eta_2\in
\mathcal{R}({B^T})$.
So we get $y^*_1-y^*_2, \mathrm R(y^*_1)-\mathrm R(y^*_2)\in
\mathcal{R}({A^T})\cap\mathcal{R}({B^T})$.
\end{proof}

With the assumption the matrix $A$ has full row rank, the following condition is
sufficient \cite{1302316}
and necessary \cite{yin} to ensure existence of a unique solution $x^*$ to
(\ref{BP}):
\begin{itemize}
 \item [1.] those columns of $A$ with respect to the support of $x^*$ are
linearly independent.
\item [2.] there exists a dual certificate $\eta\in
\partial\|x^*\|_1\cap\mathcal{R}(A^T)$ such that
$\mathbbm{P}_{\mathcal{N}(B)}(\eta)=\sgn(x^*)$
and $\|\mathbbm{P}_{\mathcal{R}(B^T)}(\eta)\|_\infty<1$.
\end{itemize}

Therefore, with assumption that there is a unique solution $x^*$ to (\ref{BP}),
there always exists a dual certificate
$\eta\in
\partial\|x^*\|_1\cap\mathcal{R}(A^T)$ such that
$\mathbbm{P}_{\mathcal{N}(B)}(\eta)=\mathbbm{P}_{\mathcal{N}(B)}(x^*)$
and $\|\mathbbm{P}_{\mathcal{R}(B^T)}(\eta)\|_\infty<1$.
By Lemma \ref{lemma1}, $y^*=x^*-\gamma \eta$ is a fixed point. And $\mathrm
R(y^*)$ is in the interior of $\mathcal{Q}$ since
$\mathrm R(y^*)=\mathrm R(x^*-\gamma \eta)=x^*+\gamma\eta$.

We call a fixed point $y^*$ an \textit{interior fixed point} if $\mathrm R(y^*)$ is
in the interior of the set $\mathcal{Q}$, or a \textit{boundary fixed point} otherwise.
A boundary fixed point exists only if $\mathcal{R}(A^T)\cap
\mathcal{R}(B^T)\neq\{\mathbf{0}\}$.

\begin{defini}
\label{def1}
Let $\mathcal{U}$ and $\mathcal{V}$ be two subspaces of $\mathbbm{R}^n$ with
$dim(\mathcal{U})=p\leq dim(\mathcal{V}) $.
The
principal angles $\theta_k\in[0,\frac{\pi}{2}]$ ($k=1,\cdots,p$) between
$\mathcal{U}$ and $\mathcal{V}$
are recursively defined by
$$\cos \theta_k=\max_{u\in \mathcal{U}}\max_{v\in \mathcal{V}}u^Tv=u^T_kv_k,
\quad \|u\|=\|v\|=1,$$
\[u^T_ju=0,\quad, u^T_ju=0, \quad, j=1,2,\cdots, k-1.\]
The vectors $(u_1,\cdots,u_p)$ and $(v_1,\cdots,v_p)$ are called principal
vectors.
\end{defini}

\begin{lmm}
\label{lemma5}
Assume $y^*$ is a boundary fixed point and $\mathrm R(y^*)$ lies on a $L$-dimensional face
of the set $\mathcal{Q}$.
Namely, there are $L$ coordinates $j_1,\cdots,j_L$
such that $|\mathrm R(y^*)_{j_l}|= \gamma$ ($l=1,\cdots,L$).
Recall that
 ${B}=[e_{i_1},\cdots,e_{i_r}]^T$, hence $\{j_1,\cdots,j_L\}$ is a subset of
$\{i_1,\cdots,i_r\}$.
Let $B_1$ denote the $(r-1)\times n$ matrix consisting of all
row vectors of $B$ except $[e_{j_1}]^T$. Recursively define $B_l$ as
the $(r-l)\times n$ matrix consisting of all
row vectors of $B_{l-1}$ except $[e_{j_l}]^T$ for $l=2,\cdots,L$.
If there exists an index $l$ such that $\mathcal{R}(A^T)\cap\mathcal{R}(B_l^T)=\mathbf{0}$, let $M$ be the smallest  such integer; otherwise, let $M=L$.
Then $M\leq \dim{\left[\mathcal{R}(A^T)\cap\mathcal{R}(B^T)\right]}$, and the first principal angle between $\mathcal{N}(A)$ and
$\mathcal{N}(B_l)$ $(l=1,\cdots,M)$ is nonzero.
    \end{lmm}
\begin{proof}
Let $\mathcal{R}_l$ ($l=1,\cdots,L$)  denote the one dimensional subspaces
spanned by $e_{j_l}$,
then $\mathcal{R}(B_{l-1})=\mathcal{R}_l\oplus\mathcal{R}(B_{l})$
and $\mathcal{N}(B_{l})=\mathcal{R}_l\oplus\mathcal{N}(B_{l-1})$.

Let $z^*$ be an interior fixed point.  Notice that $|\mathrm
R(y^*)_{j_l}|=\gamma$ and $|\mathrm R(z^*)_{j_l}|<\gamma$ for each
$l=1,\cdots,L$, thus
$\mathbbm{P}_{\mathcal{R}_l}[\mathrm R(y^*)-\mathrm R(z^*)]=\mathrm
R(y^*)_{j_l}-\mathrm R(z^*)_{j_l} \neq\mathbf{0}$.
By Lemma \ref{lemma1} we have
$\mathrm R(y^*)-\mathrm R(z^*)\in
\mathcal{R}(A^T)\cap\mathcal{R}(B^T)$, therefore
\begin{equation}
 \label{eq16}
\mathcal{R}_l\nsubseteq (\mathcal{R}(A^T)\cap\mathcal{R}(B^T))^\perp,\quad
\forall l=1,\cdots,L.
\end{equation}
Since
$\mathcal{R}(B^T)=\mathcal{R}(B_l^T)\oplus\mathcal{R}_1\oplus\cdots\oplus\mathcal{R}_{
l-1}$, with (\ref{eq16}),
we conclude that
\[
\dim{\left[\mathcal{R}(A^T)\cap\mathcal{R}(B_{1}^T)\right]}\leq\dim{
\left[\mathcal{R}(A^T)\cap\mathcal{R}(B^T)
\right]}-1.\]
Similarly, we have
\[
\dim{\left[\mathcal{R}(A^T)\cap\mathcal{R}(B_{l}^T)\right]}\leq\dim{
\left[\mathcal{R}(A^T)\cap\mathcal{R}(B_{l-1}^T)
\right]}-1,\quad l=1,\cdots,M.\]
Therefore,
 \begin{equation}
 \label{eq17}\dim{\left[\mathcal{R}(A^T)\cap\mathcal{R}(B_l^T)\right]}\leq\dim{
\left[\mathcal{R}(A^T)\cap\mathcal{R}(B^T)
\right]}-l,\quad \forall l=1,\cdots,M,\end{equation} thus $M\leq
\dim{\left[\mathcal{R}(A^T)\cap\mathcal{R}(B^T)\right]}$. 

Let $\mathcal{N}(A)\cup\mathcal{N}(B)$ denote the subspace spanned by
$\mathcal{N}(A)$ and $\mathcal{N}(B)$.
Since
$\mathbbm{R}^n=[\mathcal{R}(A^T)\cap\mathcal{R}(B^T)]\oplus[\mathcal{N}
(A)\cup\mathcal{N}(B)]
=[\mathcal{R}(A^T)\cap\mathcal{R}( B_l^T)]\oplus[\mathcal{N}(A)\cup\mathcal{N}(
B_l)]$, by(\ref{eq17}),
we have
$\dim{\left[\mathcal{N}(A)\cup\mathcal{N}( B_l)\right]}\geq\dim{
\left[\mathcal{N}(A)\cup\mathcal{N}(B)\right]}+l=\dim{[\mathcal{N}(A)]}+\dim{[
\mathcal{N}(B)]}+l=\dim{[\mathcal{N}(A)]}+\dim{[\mathcal{N}(B_l)]}$ for
$(l=1,\cdots,M)$.
Therefore $\mathcal{N}(A)\cap\mathcal{N}(B_l)=\mathbf{0}$, and the first
principal angle between $\mathcal{N}(A)$ and
$\mathcal{N}(B_l)$ is nonzero.
\end{proof}

\subsection{The characterization of the operator $T_\gamma$}
\label{subsection23}

\begin{lmm}
\label{lmm3}
 For any $y$ satisfying $\mathrm R(y)\in \mathcal{Q}$ and any fixed point $y^*$,
$T_\gamma(y)-T_\gamma(y^*)=[(I_n-B^+B)(I_n-A^+A)+B^+BA^+A](y-y^*)$ where $I_n$
denotes the $n\times n$ identity matrix.
\end{lmm}
\begin{proof}
 First, we have \begin{eqnarray*}
 T_\gamma(y)&=&[S_\gamma \circ (2\mathrm P -\mathrm{I}) +\mathrm{I}-\mathrm
P](y)
=S_\gamma(\mathrm R(y))+y-\mathrm P(y)\\
&=&
\mathrm R(y)-\gamma\sum_{j\in N(x^*)}e_j\sgn(x^*_j)-B^+B\mathrm R(y)+y-\mathrm
P(y)\\
&=& \mathrm P(y)-\gamma\sum_{j\in N(x^*)}e_j\sgn(x^*_j)-B^+B\mathrm R(y).
               \end{eqnarray*}
The last step is due to the fact $\mathrm R=2\mathrm P-\mathrm I$.
The definition of fixed points and (\ref{eq1}) imply
\begin{equation}
\label{eq5}
S_\gamma(\mathrm R(y^*))=x^*,
\end{equation}
 thus $\mathrm R(y^*)\in \mathcal{Q}$. So we also have
$$T_\gamma(y^*)=\mathrm P(y^*)-\gamma\sum_{j\in
N(x^*)}e_j\sgn(x^*_j)-B^+B\mathrm R(y^*).$$
Let $v=y-y^*$, then
\begin{eqnarray*}
T_\gamma(y)-T_\gamma(y^*)&=&\mathrm P(y)-B^+B \mathrm R(y)-\left[\mathrm
P(y^*)-B^+B \mathrm R(y^*)\right]\\
&=& y+A^+(b-Ay)-B^+B(y+2A^+(b-Ay))\\
&&-\left[y^*+A^+(b-Ay^*)-B^+B(y^*+2A^+(b-Ay^*))\right]\\
&=& v-A^+A v -B^+Bv +2B^+BA^+A v\\
&=& [(I_n-B^+B)(I_n-A^+A)+B^+BA^+A]v.
\end{eqnarray*}
\end{proof}

We now study the matrix 
\begin{equation}
 \label{matrix_T}
\mathbf{T}=(I_n-B^+B)(I_n-A^+A)+B^+BA^+A.
\end{equation}

Let $A_0$ be a $n\times(n-m)$ matrix whose column vectors form an orthonormal
basis of $\mathcal{N}(A)$ and $A_1$
be a $n\times m$ matrix whose column vectors form an orthonormal basis of
$\mathcal{R}(A^T)$. Since $A^+A$ represents the projection to $\mathcal{R}(A^T)$
and so is $A_1A_1^T$, we have $A^+A=A_1A_1^T$. Similarly, $I_n-A^+A=A_0A_0^T$.
Let $B_0$ and $B_1$ be
similarly defined for $\mathcal{N}(B)$ and $\mathcal{R}(B^T)$. The matrix
$\mathbf{T}$ can now be
written as $$\mathbf{T}=B_0B_0^TA_0A_0^T+B_1B_1^TA_1A_1^T.$$

It will be convenient to study the norm of the matrix $\mathbf{T}$ in terms of
principal angles
between subspaces.

Without loss of generality, we assume $n-r\leq n-m$. Let $\theta_i$
($i=1,\cdots,n-r$) be the principal angles between the subspaces
$\mathcal{N}(A)$
and $\mathcal{N}(B)$. Then the first principal angle $\theta_1>0$ since
$\mathcal{N}(A)\cap
\mathcal{N}(B)=\mathbf{0}$. Let $\cos \Theta$ denote the $(n-r)\times(n-r)$
diagonal matrix with the diagonal entries
$(\cos \theta_1,\cdots, \cos \theta_{(n-r)}).$

The singular value decomposition (SVD) of the $(n-r)\times(n-m)$ matrix
$E_0=B_0^TA_0$ is $E_0=U_0\cos \Theta V^T$ with
$U_0^TU_0=V^TV=I_{(n-r)}$, and the column vectors of $B_0U_0$ and $A_0V$ give
the principal vectors, see Theorem 1 in \cite{Bjoerck:1971:NMC:891913}.

By the definition of SVD, $V$ is a $(n-m)\times (n-r)$ matrix and its column
vectors are orthonormalized.
 Let $V'$ be a $(n-m)\times(r-m)$ matrix whose column vectors are normalized and
orthogonal to those of $V$.
For the matrix $\widetilde{V}=(V,V')$, we have
$I_{(n-m)}=\widetilde{V}\widetilde{V}^T$.
For the matrix $E_1=B_1^TA_0$, consider $E_1^TE_1=A_0^TB_1B_1^TA_0$. Since
$B_0B_0^T+B_1B_1^T=I_n$,
we have $E_1^TE_1=A_0^TA_0-A_0^TB_0B_0^TA_0=I_{(n-m)}-V\cos^2 \Theta V^T=(V,V')
 \left( \begin{array}{cc}
\sin^2 \Theta & 0 \\
0 & I_{(r-m)} \end{array} \right)(V,V')^T,
$
so the SVD of $E_1$ can be written as
\begin{equation}
\label{eq9}
B_1^TA_0=E_1=U_1\left( \begin{array}{cc}
\sin \Theta & 0 \\
0 & I_{(r-m)} \end{array} \right) \widetilde V^T.
\end{equation}

Notice that $A_0=B_0B_0^TA_0+B_1B_1^TA_0=B_0E_0+B_1E_1$, so we have
\begin{eqnarray}
\notag A_0A_0^T&=&(B_0,B_1)\left( \begin{array}{cc}
E_0E_0^T & E_0E_1^T \\
E_1E_0^T & E_1E_1^T \end{array} \right)(B_0,B_1)^T\\
\label{add_appendixc}&=&
(B_0U_0,B_1U_1)\left(\begin{array}{c|cc}
\cos^2 \Theta & \cos \Theta \sin \Theta & 0\\
\hline
\cos \Theta \sin \Theta  & \sin^2 \Theta &0 \\
0 & 0 & I_{(r-m)}\\
\end{array}\right)(B_0U_0,B_1U_1)^T.
\end{eqnarray}

Let $\mathcal{C}$ denote the orthogonal complement of $\mathcal{R}(A^T)\cap
\mathcal{R}(B^T)$ in
the subspace $\mathcal{R}(B^T)$, namely, $\mathcal{R}(B^T)=[\mathcal{R}(A^T)\cap
\mathcal{R}(B^T)]\oplus\mathcal{C}$.
Then the dimension of $\mathcal{C}$ is $n-m$.
 Let $\widetilde B_0=B_0U_0$ and $\widetilde B_1=B_1U_1$, then the column
vectors of
$\widetilde B_0$ form an orthonormal basis of $\mathcal{N}(B)$. The column
vectors of
$\widetilde B_1$ are a family of orthonormal vectors in $\mathcal{R}(B^T)$.
Moreover, the SVD (\ref{eq9}) implies the columns of $\widetilde B_1$ and
$A_0\widetilde V$ are principal vectors
corresponding to angles $\{\frac{\pi}{2}-\theta_1, \cdots,
\frac{\pi}{2}-\theta_{(n-r)}, 0, \cdots, 0\}$ between the two subspaces
$\mathcal{R}(B^T)$ and $\mathcal{N}(A)$ , see \cite{Bjoerck:1971:NMC:891913}.
And $\theta_1>0$ implies the largest angle between $\mathcal{R}(B^T)$ and
$\mathcal{N}(A)$ is less than
$\pi/2$, so
none of
the
  column vectors of $\widetilde B_1$ is orthogonal to $\mathcal{N}(A)$ thus
all the
  column vectors of $\widetilde B_1$ are in the subspace $\mathcal{C}$.
By counting the dimension of $\mathcal{C}$, we know that
  column vectors of $\widetilde B_1$ form an orthonormal basis of $\mathcal{C}$.

Let $\widetilde B_2$ be a $n\times(r+m-n)$ whose columns form an orthonormal
basis of $\mathcal{R}(A^T)\cap \mathcal{R}(B^T)$,
then we have
\begin{eqnarray*}
 A_0A_0^T
&=&
(\widetilde B_0, \widetilde B_1, \widetilde B_2)\left(\begin{array}{c|cc|c}
\cos^2 \Theta & \cos \Theta \sin \Theta & 0 & 0\\
\hline
\cos \Theta \sin \Theta  & \sin^2 \Theta &0 & 0\\
0 & 0 & I_{(r-m)} & 0\\
\hline
0 & 0& 0& 0_{(r+m-n)}
\end{array}\right)\left(\begin{array}{c}
\widetilde B_0^T \\\widetilde B_1^T\\ \widetilde B_2^T
\end{array}\right).
\end{eqnarray*}

Since $(\widetilde B_0, \widetilde B_1, \widetilde B_2)$
is a unitary matrix and $A_1A_1^T=I_n-A_0A_0^T$, we also have
\begin{equation}
\label{eq11}
  A_1A_1^T
=
(\widetilde B_0, \widetilde B_1, \widetilde B_2)\left(\begin{array}{c|cc|c}
\sin^2 \Theta & -\cos \Theta \sin \Theta & 0 & 0\\
\hline
-\cos \Theta \sin \Theta  & \cos^2 \Theta &0 & 0\\
0 & 0 & 0_{(r-m)} & 0\\
\hline
0 & 0& 0& I_{(r+m-n)}
\end{array}\right)\left(\begin{array}{c}
\widetilde B_0^T \\\widetilde B_1^T\\ \widetilde B_2^T
\end{array}\right).
\end{equation}

Therefore, we get the decomposition
\begin{eqnarray}
 \notag\mathbf{T}&=&B_0B_0^TA_0A_0^T+B_1B_1^TA_1A_1^T\\
\label{eq6}&=&
(\widetilde B_0, \widetilde B_1, \widetilde B_2)\left(\begin{array}{c|cc|c}
\cos^2 \Theta & \cos \Theta \sin \Theta & 0 & 0\\
\hline
-\cos \Theta \sin \Theta  & \cos^2 \Theta &0 & 0\\
0 & 0 & 0_{(r-m)} & 0\\
\hline
0 & 0& 0& I_{(r+m-n)}
\end{array}\right)\left(\begin{array}{c}
\widetilde B_0^T \\\widetilde B_1^T\\ \widetilde B_2^T
\end{array}\right).
\end{eqnarray}

\subsection{Standard cases: the interior fixed points}

Assume the sequence $y^k$ will converge to an interior fixed point.

First, consider the simple case when $\mathcal{R}(A^T)\cap
\mathcal{R}(B^T)=\{\mathbf{0}\}$, then $m+r= n$ and the fixed point is unique
and interior.
 Let $\mathcal{B}_a(z)$ denote the ball centered at $z$ with radius $a$.
Let $\varepsilon$ be the largest number such that
$\mathcal{B}_\varepsilon(\mathrm R(y^*))\subseteq \mathcal{Q}$.
 Let $K$ be the smallest integer such that $y^K\in \mathcal{B}_\varepsilon(y^*)$
(thus $\mathrm R(y^K)\in \mathcal{B}_\varepsilon(\mathrm R(y^*))$).
By nonexpansiveness of $T_\gamma$ and $\mathrm R$, we get $\mathrm R(y^k)\in
\mathcal{B}_\varepsilon(\mathrm R(y^*))$
for any $k\geq K$.
  By a recursive application of Lemma \ref{lmm3}, we have
\begin{equation*}
T_\gamma(y^{k})-y^*=\mathbf{T}(T_\gamma(y^{k-1})-y^*)=\cdots=\mathbf{T}^{k-K}(y^
{K}-y^*),\quad \forall k> K.
\end{equation*}
Now, (\ref{eq6}) and $\mathcal{R}(A^T)\cap \mathcal{R}(B^T)=\{\mathbf{0}\}$ imply
$\|\mathbf{T}\|_2=\cos\theta_1$. Notice that $\mathbf{T}$ is normal, so we have
$\|\mathbf{T}^q\|_2=\|\mathbf{T}\|_2^q$ for any positive integer $q$. Thus we
get the convergence rate for large $k$:
\begin{equation}
 \label{err1}
\|T_\gamma(y^{k})-y^*\|_2\leq(\cos\theta_1)^{k-K}\|y^{K}-y^*\|_2,\quad \forall
k> K.
\end{equation}

If $\mathcal{R}(A^T)\cap \mathcal{R}(B^T)\neq\{\mathbf{0}\}$, then there are
many fixed points  by  Lemma \ref{lemma1}.
 Let $\mathcal{I}$ be the set of
all interior fixed points.
 For $z^*\in \mathcal{I}$, let $\varepsilon(z^*)$ be the largest number such
that $\mathcal{B}_{\varepsilon(z^*)}(\mathrm R(z^*))\subseteq \mathcal{Q}$.

If $y^K\in \bigcup\limits_{z^*\in \mathcal{I}}
\mathcal{B}_{\varepsilon(z^*)}(z^*)$ for some $K$, then
consider the Euclidean projection of $y^K$ to $\mathcal{I}$, denoted by $y^*$.
Then  $\mathbbm{P}_{\mathcal{R}(A^T)\cap \mathcal{R}(B^T)}(y^K-y^*)=\mathbf{0}$
since $y^*_1-y^*_2\in \mathcal{R}(A^T)\cap \mathcal{R}(B^T)$ for any
$y^*_1,y^*_2\in\mathcal{I}$.
By (\ref{eq6}), $\mathcal{R}(A^T)\cap \mathcal{R}(B^T)$ is the eigenspace of
eigenvalue $1$ for the matrix
$\mathbf{T}$. So we have $\|\mathbf{T}(y^K-y^*)\|\leq \cos\theta_1 \|y^K-y^*\|$,
thus
the error estimate (\ref{err1}) still holds.

The sequence $y^k$ may converge to a different fixed
points for each initial value $y^0$; the fixed point $y^*$ is the projection of $y^K$ to $\mathcal{I}$. Here $K$ is the
smallest integer such that $y^K\in \bigcup\limits_{z^*\in \mathcal{I}}
\mathcal{B}_{\varepsilon(z^*)}(\mathrm R(z^*))$.

\begin{thm}
 For the algorithm (\ref{alg1}) solving (\ref{BP}), if $y^k$ converges to an
interior fixed point, then
there exists an integer $K$ such that
(\ref{err1}) holds.
\end{thm}

\subsection{Nonstandard cases: the boundary fixed points}
\label{subsection25}
Suppose $y^k$ converges to a boundary fixed point $y^*$.
  With the same notations in Lemma \ref{lemma5},  for simplicity,
we only discuss the case $M=1$. 
 More general cases can be discussed similarly.
Without loss of generality,
assume $j_1 = 1$ and $\mathrm R (y^*)_1=\gamma$. Then the set $\mathcal{Q}$ is equal to $Q_1
\oplus Q_2 \oplus\cdots \oplus Q_n$, with $Q_1 = [-\gamma, \gamma]$.
Consider another set $\mathcal{Q}_1 = (\gamma,+\infty) \oplus Q_2 \oplus\cdots\oplus
Q_n$. Any neighborhood of $\mathrm R (y^*)$ intersects both $\mathcal{Q}$ and $\mathcal{Q}_1$.

There are three cases:
\begin{itemize}
 \item [I.] the sequence $\mathrm R(y^k)$ stays in $\mathcal{Q}$ if $k$ is large
enough,
 \item [II.] the sequence $\mathrm R(y^k)$ stays in $\mathcal{Q}_1$ if $k$ is
large enough,
\item [III.] for any $K$, there exists $k_1,k_2>K$ such that $\mathrm
R(y^{k_1})\in \mathcal{Q}$ and $\mathrm R(y^{k_2})\in \mathcal{Q}_1$.
\end{itemize}

 \textbf{Case I.} Assume  $y^k$ converges to $y^*$ and  $\mathrm R(y^k)$ stay in
$\mathcal{Q}$
for any $k\geq K$.
Then
 $\mathbbm{P}_{\mathcal{R}(A^T)\cap \mathcal{R}(B^T)} (y^K - y^* )$  must be zero.
Otherwise, by
(\ref{eq6}), we have
$\lim\limits_{k\rightarrow \infty} y^k-y^*=\mathbbm{P}_{\mathcal{R}(A^T)\cap
\mathcal{R}(B^T)}(y^K-y^*)\neq \mathbf{0}$.
By (\ref{eq6}), the eigenspace of $\mathbf{T}$ associated with the eigenvalue $1$ is
$\mathcal{R}(A^T)\cap \mathcal{R}(B^T)$, so (\ref{err1}) still holds.

\textbf{Case II}
Assume  $y^k$ converges to $y^*$ and  $\mathrm R(y^k)$ stay in $\mathcal{Q}_1$
for any $k\geq K$.
Let
 $\bar{B}=[e_{i_2},\cdots,e_{i_r}]^T$.
Following Lemma \ref{lmm3}, for any $y$ satisfying $\mathrm R(y)\in
\mathcal{Q}_1$,
we have
$T_\gamma(y)-T_\gamma(y^*)=[(I_n-\bar{B}^+\bar{B})(I_n-A^+A)+\bar{B}^+\bar{B}
A^+A](y-y^*)$.

Without loss of generality,
assume $n-r+1\leq n-m$. Consider the $(n-r+1)$ principal angles between
$\mathcal{N}(A)$ and $\mathcal{N}(\bar{B})$
denoted by
$(\bar\theta_1,\cdots, \bar\theta_{(n-r+1)}).$ Let $\Theta_1$ denote the
diagonal matrix with diagonal entries
$(\bar\theta_1,\cdots, \bar\theta_{(n-r+1)}).$ Then the matrix
$\bar{\mathbf{T}}=(I_n-\bar{B}^+\bar{B})(I_n-A^+A)+\bar{B}^+\bar{B}A^+A$
can be written as
\begin{equation}
\label{eq7}
 \bar{\mathbf{T}}=
(\widetilde B_0, \widetilde B_1, \widetilde B_2)\left(\begin{array}{c|cc|c}
\cos^2 \Theta_1 & \cos \Theta_1 \sin \Theta_1 & 0 & 0\\
\hline
-\cos \Theta_1 \sin \Theta_1  & \cos^2 \Theta_1 &0 & 0\\
0 & 0 & 0_{(r-m-1)} & 0\\
\hline
0 & 0& 0& I_{(r+m-n-1)}
\end{array}\right)\left(\begin{array}{c}
\widetilde B_0^T \\\widetilde B_1^T\\ \widetilde B_2^T
\end{array}\right),
\end{equation}
where $(\widetilde B_0, \widetilde B_1, \widetilde B_2)$ are redefined
accordingly.

By Lemma \ref{lemma5}, $\bar\theta_1>0$.
Following the first case, we have  $\mathbbm{P}_{\mathcal{R}(A^T)\cap
\mathcal{R}(\bar B^T)}(y^K-y^*)=\mathbf{0}$.
So $$\|T_\gamma(y^{k})-y^*\|_2\leq(\cos\bar\theta_1)^{k-K}\|y^{K}-y^*\|_2,\quad
\forall k> K.$$
Convergence is slower than previously, as $\bar\theta_1 \leq \theta_1$.

\textbf{Case III}
Assume  $y^k$ converges to $y^*$ and  $\mathrm R(y^k)$ stay in
$\mathcal{Q}\cup\mathcal{Q}_1$
for any $k\geq K$.
Then $\mathbbm{P}_{\mathcal{R}(A^T)\cap \mathcal{R}(\bar
B^T)}(y^K-y^*)=\mathbf{0}$.
And for $y^k\in \mathcal{Q}_1$ we have
$\|T_\gamma(y^{k})-y^*\|_2\leq(\cos\bar\theta_1)\|y^k-y^*\|_2$.
Let $\mathcal{D}$ be the orthogonal complement of $\mathcal{R}(A^T)\cap
\mathcal{R}(\bar B^T)$
in $\mathcal{R}(A^T)\cap \mathcal{R}( B^T)$, namely $\mathcal{R}(A^T)\cap
\mathcal{R}( B^T)=
\mathcal{R}(A^T)\cap \mathcal{R}(\bar B^T)\oplus\mathcal{D}$. For $y^k\in
\mathcal{Q}$, we have
$\|\mathbbm{P}_{\mathcal{D}^\bot}(T_\gamma(y^{k})-y^*)\|_2\leq \cos\theta_1
\|\mathbbm{P}_{\mathcal{D}^\bot}(y^{k}-y^*)\|_2$
and
$\mathbbm{P}_{\mathcal{D}}(T_\gamma(y^{k})-y^*)=\mathbbm{P}_{\mathcal{D}}(y^{k}
-y^*)$.

For the Case III, which we refer to as nongeneric cases, no convergence results like
$\|T_\gamma(y^{k})-y^*\|_2\leq(\cos\bar\theta_1)\|y^k-y^*\|_2$ can be established since
$\mathbbm{P}_{\mathcal{D}}(T_\gamma(y^{k})-y^*)=\mathbbm{P}_{\mathcal{D}}(y^{k}
-y^*)$ whenever $\mathrm{R} (y^k)\in \mathcal{Q}$. Even though it seems hard
to exclude Case III from the analysis, it has not been observed in our numerical tests.

\subsection{Generalized Douglas-Rachford}
Consider the generalized Douglas-Rachford splitting (\ref{GDR}) with constant relaxation parameter:
\begin{equation}\label{alg5}
\begin{cases} {y}^{k+1}=  y^k+\lambda\left[S_\gamma(2 x^k -y^k) -x^k\right]\\
{x}^{k+1}={y}^{k+1}+A^+(b-Ay^{k+1})\end{cases},\quad \lambda\in(0,2).\end{equation}

Let $T_\gamma^{\lambda}=\mathrm I+\lambda\left[S_\gamma \circ (2\mathrm  P
-\mathrm{I})  -\mathrm P\right]$. Then any fixed point $y^*$ of
$T_\gamma^{\lambda}$
satisfies $\mathrm P(y^*)=x^*$, \cite{Combettes04solvingmonotone}. So the fixed
points set of $T_\gamma^{\lambda}$ is the same as the fixed points set of
$T_\gamma$. Moreover,
for any $y$ satisfying $\mathrm R(y)\in \mathcal{Q}$ and any fixed point $y^*$,
$T^{\lambda}_\gamma(y)-T^{\lambda}_\gamma(y^*)=[
I_n+\lambda(I_n-B^+B)(I_n-2A^+A)-\lambda(I_n-A^+A)](y-y^*)$.

To find the asymptotic convergence rate of (\ref{alg5}), it suffices to consider
the matrix
$\mathbf{T}_\lambda=I_n+\lambda(I_n-B^+B)(I_n-2A^+A)-\lambda(I_n-A^+A)
=(1-\lambda)I_n+\lambda\mathbf{T}$.
By (\ref{eq6}), we have
\[
 \mathbf{T}_\lambda=
\widetilde B \left(\begin{array}{c|cc|c}
\cos^2 \Theta+(1-\lambda)\sin^2 \Theta & \lambda\cos \Theta \sin \Theta & 0 &
0\\
\hline
-\lambda \cos \Theta \sin \Theta  & \cos^2 \Theta+(1-\lambda)\sin^2 \Theta &0 &
0\\
0 & 0 & (1-\lambda)I_{(r-m)} & 0\\
\hline
0 & 0& 0& I_{(r+m-n)}
\end{array}\right)\widetilde B^T,
\]
where $\widetilde B=(\widetilde B_0, \widetilde B_1, \widetilde B_2)$.

Notice that $\mathbf{T}_\lambda$ is a normal matrix. By the discussion in
Section 2, if $y^k$ in the iteration of (\ref{alg5}) converges to
an interior fixed point, the asymptotic convergence rate will be governed
by the matrix $$\mathbf{M}_\lambda=\left(\begin{array}{c|cc}
\cos^2 \Theta+(1-\lambda)\sin^2 \Theta & \lambda\cos \Theta \sin \Theta & 0 \\
\hline
-\lambda \cos \Theta \sin \Theta  & \cos^2 \Theta+(1-\lambda)\sin^2 \Theta &0 \\
0 & 0 & (1-\lambda)I_{(r-m)}
\end{array}\right).$$
Note that $\|\mathbf{M}_\lambda\|=\sqrt{\lambda(2-\lambda)\cos^2
\theta_1+(1-\lambda)^2}\geq \cos\theta_1$ for any $\lambda\in(0,2)$. Therefore,
the asymptotic convergence rate of (\ref{alg5}) is always slower than
(\ref{alg2}) if $\lambda\neq 1$.
We emphasize that this does not mean (\ref{alg2}) is more efficient
than (\ref{alg5}) for $x^k$ to reach a given accuracy.

\subsection{Relation to the Restricted Isometry Property}
Let $A$ be a $m\times n$ random matrix and each column of $A$ is normalized,
i.e.,
$\sum\limits_{i}A_{ij}^2=1$ for each $j$. The Restricted Isometry Property (RIP)
introduced in \cite{1542412} is as follows.
\begin{defini}
For each integer $s=1,2,\cdots,$ the restricted isometry constants $\delta_s$ of
$A$    is the smallest number such that
\begin{equation}
\label{eq8}
(1-\delta_s)\|x\|^2\leq \|Ax\|^2\leq (1+\delta_s)\|x\|^2,
\end{equation}
holds for all vectors $x$ with at most $s$ nonzero entries.
\end{defini}

In particular, any vector with the same support as $x^*$ can be denoted as
$(I_n-B^+B)x$ for some $x\in \mathbbm{R}^n$. The RIP (\ref{eq8}) with $s=n-r$
implies
$$(1-\delta_{(n-r)})\|(I_n-B^+B)x\|^2\leq \|A(I_n-B^+B)x\|^2\leq
(1+\delta_{(n-r)})\|(I_n-B^+B)x\|^2,\quad \forall x\in\mathbbm{R}^n.$$

Let $d$ denote the smallest eigenvalue of $(AA^T)^{-1}$. Then $d>0$ since we
assume $A$ has full row rank.
For any vector $y$, we have
\[\|A^+Ay\|^2=y^TA^T[(AA^T)^{-1}]^TAA^T(AA^T)^{-1}Ay=y^TA^T[(AA^T)^{-1}]^TAy\geq
d\|Ay\|^2,\]
where the last step is due to the Courant–Fischer–Weyl min-max principle.

Therefore, we get
\begin{equation}
\label{eq10}
\|A^+A(I_n-B^+B)x\|^2\geq d\|A(I_n-B^+B)x\|^2\geq d
(1-\delta_{(n-r)})\|(I_n-B^+B)x\|^2, \quad \forall x\in\mathbbm{R}^n,
\end{equation}

We will show that (\ref{eq10}) gives a lower bound of the first principal angle
$\theta_1$ between two subspaces $\mathcal{N}(A)$
and $\mathcal{N}(B)$.
Notice that (\ref{eq11}) implies
\[A^+A(I_n-B^+B)=A_1A_1^TB_0B_0^T=(B_0U_0,B_1U_1)\left(\begin{array}{c|c}
\sin^2 \Theta & 0\\
\hline
-\cos \Theta \sin \Theta  & 0 \\
0 & 0 \\
\end{array}\right)(B_0U_0,B_1U_1)^T,\]
by which we have
$\|A^+A(I_n-B^+B)x\|^2=x^T(B_0U_0,B_1U_1)\left(\begin{array}{c|c}
\sin^2 \Theta & 0\\
\hline
0 & 0 \\
\end{array}\right)(B_0U_0,B_1U_1)^T x.$

Let $z=(B_0U_0,B_1U_1)^T x$. Since
$I_n-B^+B=(B_0U_0,B_1U_1)\left(\begin{array}{c|c}
I_{(n-r)} & 0\\
\hline
0 & 0\\
\end{array}\right)(B_0U_0,B_1U_1)^T$,
 (\ref{eq10}) is equivalent to
\[z^T\left(\begin{array}{c|c}
\sin^2 \Theta & 0\\
\hline
0 & 0 \\
\end{array}\right) z\geq d (1-\delta_{(n-r)}) z^T\left(\begin{array}{c|c}
I_{n-r} & 0\\
\hline
0 & 0 \\
\end{array}\right)z,\quad \forall z\in \mathbbm{R}^n,\]
which  implies $\sin^2\theta_1\geq d (1-\delta_{(n-r)})$ by the
Courant–Fischer–Weyl min-max principle.
So the RIP constant gives us
\[\cos\theta_1\leq \sqrt{1-d (1-\delta_{(n-r)})}.\]

\subsection{Numerical examples}

We consider several examples for (\ref{alg2}). In all the examples,
$y^0=\mathbf{0}$ unless specified otherwise. For examples in this subsection, the angles between the null spaces
can be computed by singular value decomposition (SVD) of $A_0^TB_0$ \cite{Bjoerck:1971:NMC:891913}.

\bigskip \noindent{\bf Example 1}
The matrix ${A}$ is a $3\times 40$ random matrix with standard normal distribution and $x^*$ has three nonzero
components. By counting dimensions, we
know that $\mathcal{R}(A^T)\cap \mathcal{R}(B^T)=\{\mathbf{0}\}$. Therefore
there is only one fixed point.
See Figure \ref{ex1} for the error curve of $x^k$ and $y^k$ with $\gamma=1$.
Obviously, the error $\|x^k-x^*\|$ is not monotonically decreasing but
$\|y^k-y^*\|$ is since the operator $T_\gamma$ is non-expansive. And the slope
of $\log \|y^k-y^*\|$ is exactly $\log(\cos\theta_1)=\log (0.9932)$ for large $k$.

\begin{figure}[ht]
  \begin{center}
\includegraphics[scale=0.5]{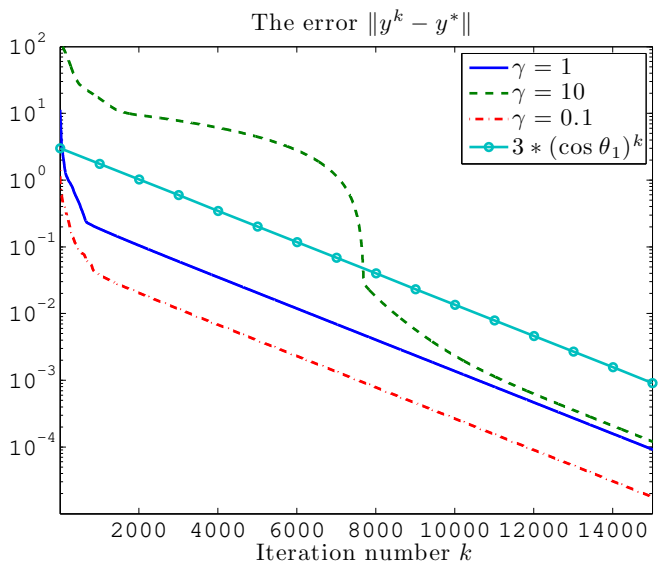}
  \end{center}
  \caption{Example 2: for an interior fixed point, the asymptotic rate remains
the same for different soft-thresholding parameter $\gamma$. The slope of the straight line is $\log(\cos \theta_1)$.}
  \label{ex2}
\end{figure}

\bigskip \noindent{\bf Example 2}
The matrix ${A}$ is a $10\times 1000$ random matrix with standard normal distribution and $x^*$ has ten nonzero
components. Thus there is only one fixed point.
See Figure \ref{ex2} for the error curve of $y^k$ with $\gamma=0.1, 1, 10$.  We
take ${y^*}$ as the result of (\ref{alg2}) after $8\times10^4$
iterations.
The slopes of $\log \|y^k-y^*\|$ for different $\gamma $ are exactly
$\log(\cos\theta_1)=\log(0.9995)$ for large $k$.

\bigskip \noindent{\bf Example 3}
The matrix ${A}$ is a $18\times 100$ submatrix of a $100\times 100$ Fourier
matrix  and $x^*$ has two nonzero components. There are interior and
boundary fixed points. In this example, we fix $\gamma=1$ and test (\ref{alg2})
with random $y^0$ for six times.
See Figure \ref{ex3} for the error curve of $x^k$.  In Figure  \ref{ex3}, in four tests, 
$y^k$ converges to an interior fix point, thus
the convergence rate  for large $k$ is governed by $\cos\theta_1=0.9163$. In the
second and third tests, $y^k$ converges to different boundary fixed points\footnote{At least, numerically so in double precision.} thus
convergence rates are slower than $\cos\theta_1$. Nonetheless, the rate for
large $k$ is still linear.

\begin{figure}[ht]
  \begin{center}
\includegraphics[scale=0.4]{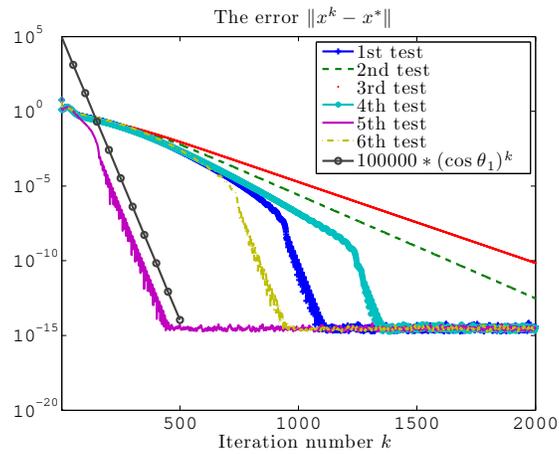}
  \end{center}
  \caption{Example 3: fixed $\gamma=1$ with random $y^0$.}
  \label{ex3}
\end{figure}
\begin{figure}[ht]
  \begin{center}
\includegraphics[scale=0.4]{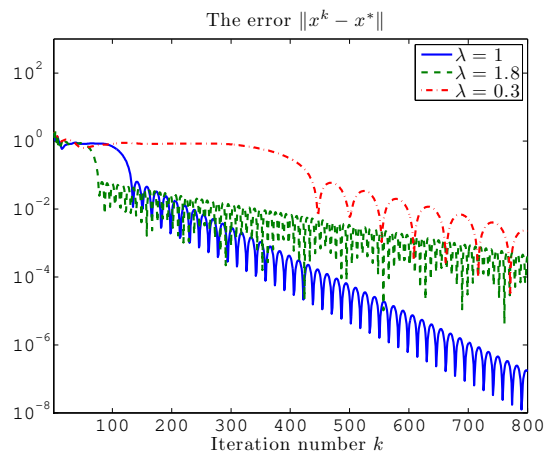}
  \end{center}
  \caption{Example 4: The Generalized Douglas-Rachford (\ref{alg5}) with
different $\lambda$ and fixed $\gamma=1$. The asymptotic convergence rate of
(\ref{alg2}) ($\lambda=1$) is the fastest. }
  \label{ex4}
\end{figure}
\bigskip \noindent{\bf Example 4}
The matrix ${A}$ is a $5\times 40$ random matrix with standard normal distribution and $x^*$ has three nonzero
components.
See Figure \ref{ex4} for the comparison of (\ref{alg2}) and (\ref{alg5}) with
$\gamma=1$.

\begin{rem}
To apply Douglas-Rachford splitting (\ref{alg1}) to basis pursuit (\ref{BP}), we can also choose $g(x) = \| x \|_1$ and $f(x) = \iota_{\{x:Ax=b\}}$, then
Douglas-Rachford iterations become
\begin{equation}\label{alg_appendix}
\begin{cases} {y}^{k+1}=   x^k + 
A^+(b-A(2x^k-y^k))\\
{x}^{k+1}=S_\gamma({y}^{k+1})
\end{cases}.
\end{equation}
The discussion in this section can be applied to (\ref{alg_appendix}). In particular, the corresponding matrix in (\ref{matrix_T})
is 
$\mathbf{T}=(I_n-A^+A)(I_n-B^+B)+A^+AB^+B$, thus all the asymptotic convergence rates remain valid. For all the numerical tests in this paper, we did not observe any significant difference in performance between (\ref{alg2}) and (\ref{alg_appendix}).

\end{rem}

\section{The $\ell^2$ regularized Basis Pursuit}
\setcounter{equation}{0}
\setcounter{figure}{0}
\setcounter{table}{0}
\label{sec3}

\subsection{Preliminaries}
For the $\ell^2$ regularized Basis Pursuit (\ref{BP2}),
to use Douglas-Rachford splitting (\ref{alg1}) to solve the equivalent problem $\min\limits_{x}
\|{x}\|_1+\iota_{\{x:Ax=b\}}+\frac{1}{2\alpha}\|x\|^2$, there are quite a few splitting choices:
\begin{enumerate}
\item \beq
\label{choice1}
f(x)=\|{x}\|_1+\frac{1}{2\alpha p}\|x\|^2, \quad g(x)=\iota_{\{x:Ax=b\}}+\frac{1}{2\alpha q}\|x\|^2, \quad \forall p,q\geq 1, \frac{1}{p}+\frac{1}{q}=1.
\eeq

\item \beq
\label{choice2}
g(x)=\|{x}\|_1+\frac{1}{2\alpha p}\|x\|^2, \quad f(x)=\iota_{\{x:Ax=b\}}+\frac{1}{2\alpha q}\|x\|^2, \quad \forall p,q\geq 1, \frac{1}{p}+\frac{1}{q}=1.
\eeq
\end{enumerate}

The following two resolvents will be needed: 
\begin{itemize}
\item $h(x)=\|{x}\|_1+\frac{1}{2\alpha}\|x\|^2$, $J_{\gamma \partial h}(x)=\argmin\limits_z \gamma \|z\|_1+
\frac{\gamma}{2\alpha}\|z\|^2+\frac12\|z-x\|^2
=\frac{\alpha}{\alpha+\gamma}S_\gamma(x)$.
\item $h(x)=\iota_{\{x:Ax=b\}}+\frac{1}{2\alpha}\|x\|^2$, $J_{\gamma \partial h}(x)=\argmin\limits_z \gamma \iota_{\{z:Az=b\}}+
\frac{\gamma}{2\alpha}\|z\|^2+\frac12\|z-x\|^2
=\frac{\alpha}{\alpha+\gamma}x+A^+(b-\frac{\alpha}{\alpha+\gamma}Ax).$
\end{itemize}

\subsection{Douglas-Rachford splitting}
\label{sec3-2}

In particular, Douglas-Rachford splitting (\ref{alg1}) using (\ref{choice1}) with $p=1$ and $q=\infty$ is equivalent to the dual split Bregman method \cite{LBSB}. See Section \ref{seclbsb} for the equivalence.  We first discuss this special case.

Let
$f(x)=\|{x}\|_1+\frac{1}{2\alpha}\|x\|^2$ and $g(x)=\iota_{\{x:Ax=b\}}$,
the Douglas-Rachford splitting (\ref{alg1}) for (\ref{BP}) reads
\begin{equation}\label{alg3}
\begin{cases} {y}^{k+1}=  \frac{\alpha}{\alpha+\gamma}S_\gamma(2 x^k -y^k) +y^k
-x^k\\
{x}^{k+1}={y}^{k+1}+A^+(b-Ay^{k+1})\end{cases}.\end{equation}

 Since $\|{x}\|_1+\frac{1}{2\alpha}\|x\|^2$
is a strongly convex function, (\ref{BP2}) always has a unique minimizer $x^*$
as long as $\{x:Ax=b\}$ is nonempty.
The first order optimality condition $\mathbf{0}\in \partial F(x^*)+\partial
G(x^*)$ implies the dual certificate set
$(\partial\|x^*\|_1+\frac{1}{\alpha}x^*)\cap\mathcal{R}(A^T)$
is nonempty.
Let $T^\alpha_\gamma=\frac{\alpha}{\alpha+\gamma}S_\gamma \circ (2\mathrm P
-\mathrm{I}) +\mathrm{I}-\mathrm P$.

 \begin{lmm}
\label{lemma3}
 The set of the fixed points of  $T^\alpha_\gamma$ can be described as
$$\left\{y^*: y^*=x^*-\gamma \eta, \eta\in
\left(\partial\|x^*\|_1+\frac{1}{\alpha}x^*\right)\cap\mathcal{R}(A^T)\right\}
.$$
\end{lmm}
The proof is similar to the one of Lemma \ref{lemma1}.
We also have
\begin{lmm}
\label{lmm4}
 For any $y$ satisfying $\frac{\alpha}{\alpha+\gamma}\mathrm R(y)\in
\mathcal{Q}$ and any fixed point $y^*$,
$T^\alpha_\gamma(y)-T^\alpha_\gamma(y^*)=\left[
c(I_n-B^+B)(I_n-A^+A)+cB^+BA^+A+(1-c)A^+A\right](y-y^*)$
where $c=\frac{\alpha}{\alpha+\gamma}$.
\end{lmm}
\begin{proof}
 First, we have \begin{eqnarray*}
 T^\alpha_\gamma(y)&=&[c S_\gamma \circ (2\mathrm P -\mathrm{I})
+\mathrm{I}-\mathrm P](y)
=c S_\gamma(\mathrm R(y))+y-\mathrm P(y)\\
&=&
c\left[\mathrm R(y)-\gamma\sum_{j\in N(x^*)}e_j\sgn(x^*_j)-B^+B\mathrm
R(y)\right]+y-\mathrm P(y)
               \end{eqnarray*}
Similarly we also have
$$T^\alpha_\gamma(y^*)=c\left[\mathrm R(y^*)-\gamma\sum_{j\in
N(x^*)}e_j\sgn(x^*_j)-B^+B \mathrm R(y^*)\right]+y^*-\mathrm P(y^*).$$
Let $v=y-y^*$, then
\begin{eqnarray*}
T^\alpha_\gamma(y)-T^\alpha_\gamma(y^*)&=&
c\left[\mathrm R(y)-B^+B \mathrm R(y)\right]+y-\mathrm P(y)\\
&&-c\left[\mathrm R(y^*)-B^+B\mathrm R(y^*)\right]-(y^*-\mathrm P(y^*))\\
&=& c[I_n-2A^+A-B^+B+2B^+BA^+A]v+A^+Av\\
&=& [c(I_n-B^+B)(I_n-A^+A)+cB^+BA^+A+(1-c)A^+A]v.
\end{eqnarray*}
\end{proof}

Consider the matrix \beq \label{matrixT}\mathbf{T}(c)=c(I_n-B^+B)(I_n-A^+A)+cB^+BA^+A+(1-c)A^+A,
\quad c=\frac{\alpha}{\alpha+\gamma}.\eeq
Then $\mathbf{T}(c)=c\mathbf{T}+(1-c)A^+A$
where $\mathbf{T}=(I_n-B^+B)(I_n-A^+A)+B^+BA^+A$.

By (\ref{eq11}) and (\ref{eq6}), we have
\begin{equation}
\label{eq12}
  \mathbf{T}(c)
=
\widetilde B\left(\begin{array}{c|cc|c}
 (1-c)\sin^2\Theta +c\cos^2 \Theta & (2c-1)\cos \Theta \sin \Theta & 0 & 0\\
\hline
-\cos \Theta \sin \Theta  &\cos^2 \Theta &0 & 0\\
0 & 0 & 0_{(r-m)} & 0\\
\hline
0 & 0& 0& I_{(r+m-n)}
\end{array}\right)\widetilde B^T,
\end{equation}
where $\widetilde B=(\widetilde B_0, \widetilde B_1, \widetilde B_2)$.

Following the proof in \cite{yin}, it is straightforward to show there exists a
dual certificate $\eta\in
(\partial\|x^*\|_1+\frac{1}{\alpha}x^*)\cap\mathcal{R}(A^T)$ such that
$\mathbbm{P}_{\mathcal{N}(B)}(\eta)=\mathbbm{P}_{\mathcal{N}(B)}(x^*)$
and $\|\mathbbm{P}_{\mathcal{R}(B^T)}(\eta)\|_\infty<1$. So there is at least
one
interior fixed point. Following Lemma \ref{lemma1},
there is only one fixed point $y^*$ if and only if $\mathcal{R}(A^T)\cap
\mathcal{R}(B^T)=\{\mathbf{0}\}$.

For simplicity, we only discuss the interior fixed point case. The boundary
fixed point case
is similar to the previous discussion.

Assume $y^k$ converges to an interior fixed point $y^*$.
Let $\varepsilon$ be the largest number such that
$\mathcal{B}_\varepsilon(\mathrm R(y^*))\subseteq S$.
 Let $K$ be the smallest integer such that $y^K\in \mathcal{B}_\varepsilon(y^*)$
(thus $\mathrm R(y^K)\in \mathcal{B}_\varepsilon(\mathrm R(y^*))$).
By nonexpansiveness of $T^\alpha_\gamma$ and $\mathrm R$, we get $\mathrm
R(y^k)\in \mathcal{B}_\varepsilon(\mathrm R(y^*))$
for any $k\geq K$. So we have
\[\|y^k-y^*\|=\|\mathbf{T}(c)(y^k-y^*)\|=\cdots=\|\mathbf{T}(c)^{k-K}
(y^K-y^*)\|\leq \|\mathbf{T}(c)^{k-K}\|\|(y^K-y^*)\|,\quad\forall k>K. \]

Notice that $\mathbf{T}(c)$ is a nonnormal matrix, so $\|\mathbf{T}(c)^k\|$ is
much less than $\|\mathbf{T}(c)\|^k$ for large $k$. Thus
the asymptotic convergence rate is governed by
$\lim\limits_{k\rightarrow\infty}\sqrt[k]{\|\mathbf{T}(c)^k\|}$, which is equal
to the norm of the eigenvalues of $\mathbf{T}(c)$ with the largest magnitude.

It suffices to study the matrix $\mathbf{M}(c)=\left(\begin{array}{c|c}
(1-c)\sin^2\Theta +c\cos^2 \Theta & (2c-1)\cos \Theta \sin \Theta\\
\hline
-\cos \Theta \sin \Theta  & \cos^2 \Theta\end{array}\right)$
because $\mathbbm{P}_{\mathcal{R}(A^T)\cap\mathcal{R}(B^T)}(y^K-y^*)=\mathbf{0}$
(otherwise $y^k$ cannot converge to $y^*$).

Notice that $\det(\mathbf{M}(c)-\rho
\mathrm{I})=\prod\limits_{i=1}^{n-r}\left[
\rho^2-(c\cos(2\theta_i)+1)\rho+c\cos^2\theta_i\right]$.
Let $\rho(\theta,c)$ denote the magnitude of the solution with the largest magnitude for the
quadratic equation $\rho^2-(c\cos(2\theta)+1)\rho+c\cos^2\theta$,
with discriminant $\Delta=\cos^2(2\theta)c^2-2c+1$.

The two solutions of $\Delta=0$ are $[1\pm\sin(2\theta)]/\cos^2(2\theta)$.
Notice that  $[1+\sin(2\theta)]/\cos^2(2\theta)\geq 1$ for $\theta\in[0,\pi/2]$
and $c\in(0,1)$, we have
\begin{equation}
\label{eq13}
 \rho(\theta,c)=\begin{cases}
\sqrt{c}\cos\theta, & \mbox{if } c\geq
\frac{1-\sin(2\theta)}{\cos^2(2\theta)}=\frac{1}{(\cos\theta+\sin\theta)^2}\\
\frac12\left(c\cos(2\theta)+1+\sqrt{\cos^2(2\theta)c^2-2c+1} \right)& \mbox{if }
c\leq \frac{1}{(\cos\theta+\sin\theta)^2}
                   \end{cases}.
\end{equation}

It is straightforward to check that $\rho(\theta,c)$ is monotonically
decreasing with respect to $\theta$ for $\theta\in[0,\frac{\pi}{4}]$. Therefore, the
asymptotic convergence rate is equal to $\rho(\theta_1,c)$ if $\theta_1\leq\frac{\pi}{4}$.

Let $c^*=\frac{1}{(\cos\theta_1+\sin\theta_1)^2}$ which is equal to
$\argmin_c\rho(\theta_1,c)$. Let $c^\sharp=\frac{1}{1+2\cos\theta_1}$
which is the solution to $\rho(\theta_1,c)=\cos\theta_1$. See Figure
\ref{bestc}. Then for any $c\in (c^\sharp,1)$, we have
$\rho(\theta_1,c)<\cos\theta_1
$. Namely, the asymptotic convergence rate of (\ref{alg3}) is faster than
(\ref{alg2}) if $\frac{\alpha}{\alpha+\gamma}\in (c^\sharp,1)$.
The best asymptotic convergence rate that (\ref{alg3}) can achieve is
$\rho(\theta_1,c^*)=\sqrt{c^*}\cos\theta_1=\frac{\cos\theta_1}{
\cos\theta_1+\sin\theta_1} = \frac{1}{1 + \tan \theta_1}$ when $\frac{\alpha}{\alpha+\gamma}=c^*$.

\begin{rem}
The general cases of the two alternatives (\ref{choice1}) and (\ref{choice2}) with any $p$ and $q$ can be discussed similarly.
For Douglas-Rachford splitting (\ref{alg1}) using (\ref{choice1}) with $q=1$ and (\ref{choice2}) with $p=1$ or $q=1$, the asymptotic linear rate (\ref{eq13}) holds.
Compared to (\ref{alg3}), we observed no improvement in numerical performance by using (\ref{choice1}) or (\ref{choice2}) with any other values of $p$ and $q$ in all our numerical tests.
\end{rem}

\subsection{Generalized Douglas-Rachford and Peaceman-Rachford splittings}
\label{sec3-3}

For the generalized Douglas-Rachford splitting
(\ref{GDR}),
 the choice of $p$ and $q$ in the (\ref{choice1}) and (\ref{choice2}) may result in different performance.
The main difference can be seen in the limiting case $\lambda_k\equiv 2$, for which (\ref{GDR}) becomes the Peaceman-Rachford splitting (\ref{PR}).

If $f(x)$ is convex and $g(x)$ is strongly convex,
the convergence of (\ref{PR}) is guaranteed, 
see \cite{combettes2009iterative,han2012convergence}. On the other hand, (\ref{GDR}) may not converge if $g(x)$ is only convex rather than strongly convex. For instance,
(\ref{PR}) with (\ref{choice1}) and $p=1$ (or (\ref{choice2}) and $q=1$) did not converge for examples in Section \ref{sec44}.
To this end, the best choices of $p$ and $q$ for (\ref{GDR}) should be (\ref{choice1}) with $q=1$ and (\ref{choice2}) with $p=1$.
We only discuss the case of using (\ref{choice1}) with $q=1$. The analysis will hold for the other one.

Let $f(x)=\|x\|_1$ and $g(x)=\iota_{\{x:Ax=b\}}+\frac{1}{2\alpha}\|x\|^2$. Consider the following generalized Douglas-Rachford splitting
with a constant relaxation parameter $\lambda$:
\beq
\label{GDR2}
\begin{cases} {y}^{k+1}=  y^k+\lambda \left[S_\gamma(2 x^k -y^k) 
-x^k\right]\\
{x}^{k+1}=\frac{\alpha}{\alpha+\gamma}{y}^{k+1}+A^+(b-\frac{\alpha}{\alpha+\gamma}Ay^{k+1})\end{cases},\quad \lambda\in (0,2].
\eeq

For the algorithm (\ref{GDR2}), the corresponding matrix in (\ref{matrixT}) is
\[\mathbf{T}(c,\lambda)=I+\lambda[(I-B^+B)(2c(I-A^+A)-I)-c(I-A^+A)]=(1-\lambda)I+\lambda[c\mathbf{T}+(1-c)B^+B],\]
where $c=\frac{\alpha}{\alpha+\lambda}$ and $\mathbf{T}=(I-B^+B)(I-A^+A)+B^+BA^+A$.

By (\ref{eq6}), we have
\begin{equation}
\label{GDR_eig}
  \mathbf{T}(c,\lambda)
=
\widetilde B\left(\begin{array}{c|cc|c}
 \lambda c\cos^2 \Theta &  \lambda c\cos \Theta \sin \Theta & 0 & 0\\
\hline
- \lambda c \cos \Theta \sin \Theta  &  \lambda c\cos^2 \Theta +(1- \lambda c)I_{(n-r)}&0 & 0\\
0 & 0 & (1- \lambda c)I_{(r-m)} & 0\\
\hline
0 & 0& 0& I_{(r+m-n)}
\end{array}\right)\widetilde B^T,
\end{equation}
where $\widetilde B=(\widetilde B_0, \widetilde B_1, \widetilde B_2)$.

It suffices to study the matrix $\mathbf{M}(c,\lambda)=\left(\begin{array}{c|c}
\lambda c \cos^2 \Theta & \lambda c\cos \Theta \sin \Theta\\
\hline
-\lambda c \cos \Theta \sin \Theta  &  \lambda c\cos^2 \Theta +(1- \lambda c)I_{(n-r)}\end{array}\right)$.
Notice that $\det(\mathbf{M}(c,\lambda)-\rho
\mathrm{I})=\prod\limits_{i=1}^{n-r}[
\rho^2-(\lambda c\cos(2\theta_i)-\lambda+2)\rho+c\sin^2\theta_i\lambda^2-(1-c\cos(2\theta_i))\lambda+1]$.
Let $\rho(\theta,c,\lambda)$ denote the magnitude of the solution with the largest magnitude for the
quadratic equation $\rho^2-(\lambda c\cos(2\theta)-\lambda+2)\rho+c\sin^2\theta\lambda^2-(1-c\cos(2\theta))\lambda+1$,
with discriminant $\Delta=\lambda^2(\cos^2(2\theta)c^2-2c+1)$.

The two solutions of $\Delta=0$ are $[1\pm\sin(2\theta)]/\cos^2(2\theta)$.
Notice that  $[1+\sin(2\theta)]/\cos^2(2\theta)\geq 1$ for $\theta\in[0,\pi/2]$
and $c\in(0,1)$, we have
\begin{equation}
\label{besteig_PR}
 \rho(\theta,c,\lambda)=\begin{cases}
\sqrt{c\sin^2\theta\lambda^2-(1-c\cos(2\theta))\lambda+1}, & \mbox{if } c\geq
\frac{1-\sin(2\theta)}{\cos^2(2\theta)}=\frac{1}{(\cos\theta+\sin\theta)^2}\\
\frac12\left(\lambda c\cos(2\theta)-\lambda+2+\lambda\sqrt{\cos^2(2\theta)c^2-2c+1} \right)& \mbox{if }
c\leq \frac{1}{(\cos\theta+\sin\theta)^2}
                   \end{cases}.
\end{equation}

It is straightforward to check that $\rho(\theta,c,\lambda)\geq |1-\lambda c|$ and $\rho(\theta,c,\lambda)$ is monotonically
decreasing with respect to $\theta$ for $\theta\in[0,\frac{\pi}{4}]$. Therefore, the
asymptotic convergence rate of (\ref{GDR2}) is governed by $\rho(\theta_1,c,\lambda)$ if $\theta_1\leq\frac{\pi}{4}$.

The next step is to evaluate $\argmin\limits_\lambda\rho(\theta,c,\lambda)$.
When $c\leq c^*= \frac{1}{(\cos\theta+\sin\theta)^2}$, $\rho(\theta,c,\lambda)$ is monotonically
decreasing with respect to $\lambda$. Let $\bar c=\frac{1}{2-\cos(2\theta)}$, for the quadratic equation  
$\kappa(\lambda)=c\sin^2\theta\lambda^2-(1-c\cos(2\theta))\lambda+1$,
 we have
\[\argmin\limits_\lambda\kappa(\lambda)=\begin{cases}2 , & \mbox{if } c^*\leq c\leq
\bar c\\ \frac{1-c\cos{2\theta}}{c(1-\cos{(2\theta)})},  & \mbox{if } \bar c\leq c<1\end{cases}
\quad \mbox{and} \quad \min\kappa(\lambda)=
\begin{cases}2c-1 , & \mbox{if } c^*\leq c\leq
\bar c\\ \frac{2c-1-c^2\cos^2{2\theta}}{4c\sin^2{\theta}},  & \mbox{if } \bar c\leq c<1\end{cases}.
\]

Let $\lambda^*(\theta,c)=\argmin\limits_\lambda\rho(\theta,c,\lambda)$, then
\beq\label{bestrelax}\lambda^*(\theta,c)=\begin{cases}2  & \mbox{if } c\leq
\bar c= \frac{1}{2-\cos(2\theta)} \\  \frac{\frac{1}{c}-\cos{2\theta}}{1-\cos{(2\theta)}} & \mbox{if } c\geq \bar c \end{cases},\eeq
which is a continuous non-increasing function w.r.t $c$ and has range $(1,2]$ for $c\in(0,1)$.

The convergence rate with $\lambda^*$ is
\[\rho(\theta,c,\lambda^*)=\begin{cases}\rho(\theta,c,2)=c\cos(2\theta)+\sqrt{\cos^2(2\theta)c^2-2c+1}, & \mbox{if } c\leq
c^*=\frac{1}{(\cos\theta+\sin\theta)^2}\\ 
\rho(\theta,c,2)=\sqrt{2c-1}, & \mbox{if } c^*\leq c\leq \bar c= \frac{1}{2-\cos(2\theta)}\\
\rho(\theta,c,\frac{1-c\cos{2\theta}}{c(1-\cos{2\theta})})=\frac{\sqrt{2c-1-c^2\cos^2{(2\theta)}}}{2\sin\theta\sqrt{c}},
& \mbox{if } c\geq \bar c
\end{cases}.\]

See Figure \ref{bestc2} for the illustration of the asymptotic linear rate $\rho(\theta,c,\lambda)$.
\begin{rem}
We emphasize several interesting facts:
\begin{itemize}
\item For Peaceman-Rachford splitting, i.e., (\ref{GDR2}) with $\lambda=2$, if $c\geq c^*$, the asymptotic rate $\rho(\theta,c,2)=\sqrt{2c-1}$ is independent of $\theta$.
\item For any  $c< \tilde c=\frac{1}{2-\cos^2\theta}$, the Peaceman-Rachford splitting is faster than Douglas-Rachford, i.e., $\rho(\theta,c,2)
<\rho(\theta,c,1)$.  
\item The best possible rate of (\ref{GDR2}) is $\rho(\theta,c^*,2)=\frac{1-\tan\theta}{1+\tan\theta}$.
\item The quadratic function $\kappa(\lambda)$ is monotonically increasing if $\lambda\geq \frac{\frac{1}{c}-\cos{2\theta}}{1-\cos{(2\theta)}}$
and decreasing otherwise. 
For any $\lambda<1$,  (\ref{besteig_PR}) and (\ref{bestrelax}) implies $\rho(\theta,c,\lambda)> \rho(\theta,c,1)$. Thus (\ref{GDR2}) with $\lambda<1$ has slower asymptotic rate than (\ref{alg3}).
\end{itemize}

\end{rem}

\begin{figure}[htbp]
  \begin{center}
\subfigure[For the algorithm (\ref{alg3}), $c^*=0.756$ indeed gives the best asymptotic rate
$\frac{1}{1+\tan \theta_1}$ but $c^*$ is not necessarily the most efficient
choice for a given accuracy.]{\includegraphics[scale=0.6]{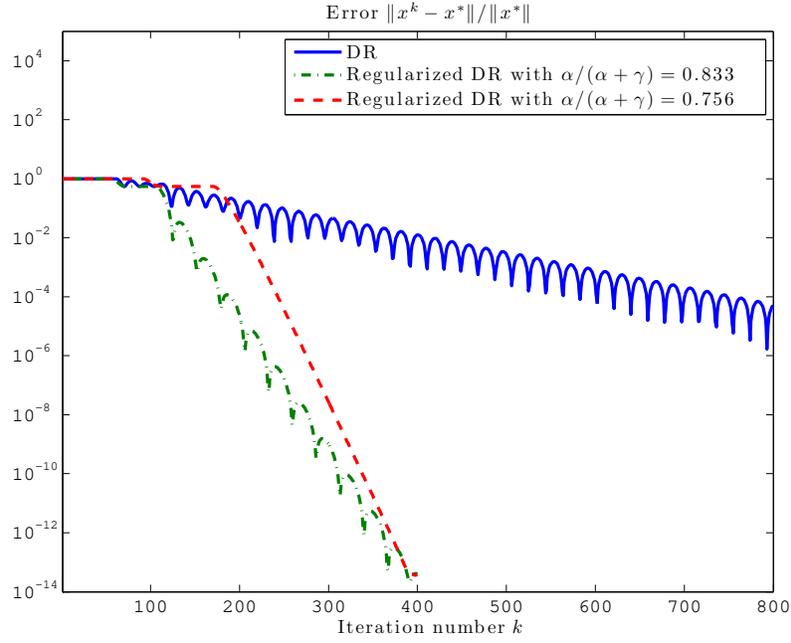}}
\subfigure[The best asymptotic rates.]{\includegraphics[scale=0.6]{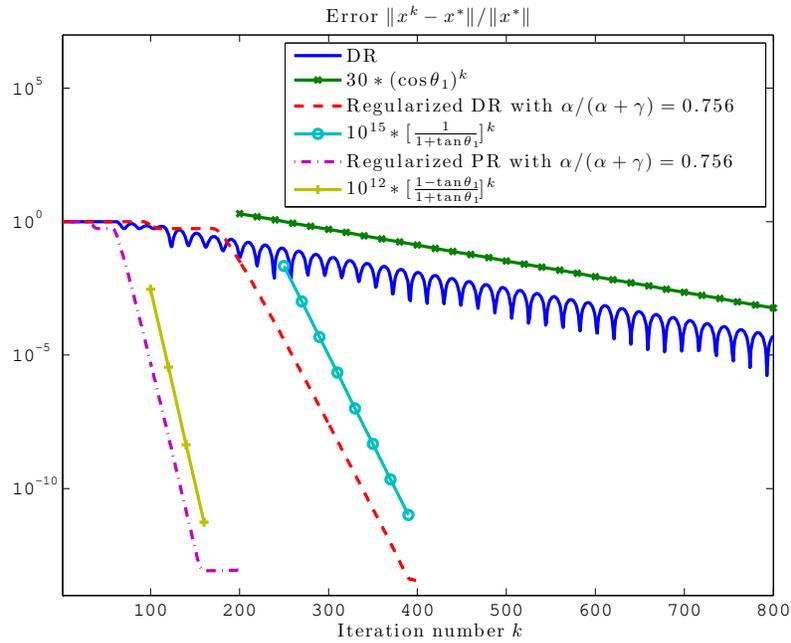}}
  \end{center}
  \caption{Example 5: $\alpha=20$ is fixed. DR stands for (\ref{alg2}) and
Regularized DR stands for (\ref{alg3}). Regularized PR stands for (\ref{GDR2}) with $\lambda=2$.}
  \label{ex5}
\end{figure}
\bigskip \noindent{\bf Example 5}
The matrix ${A}$ is a $40\times 1000$ random matrix with standard normal distribution and $x^*$ has two nonzero
components. 
We test the algorithms (\ref{alg3}) and (\ref{GDR2}). See Section \ref{seclbsb} for the equivalence between
(\ref{alg3}) and the dual split Bregman method in
\cite{LBSB}.
See Figure \ref{ex5} for the error curve of $x^k$. The best choice of the
parameter $c=\alpha/(\alpha+\gamma)$ according to Figure \ref{bestc} should be $\alpha/(\alpha+\gamma)=c^*$,
which is $c^*=0.756$
for this example. Here $c^*$ indeed gives the best asymptotic rate
$\frac{1}{1+\tan \theta_1}$ for (\ref{alg3}) but $c^*$ is not necessarily the most efficient
choice for a given accuracy, as we can see in the Figure \ref{bestc} (a).
The best asymptotic rates  (\ref{alg3}) and (\ref{GDR2})
are $\frac{1}{1+\tan\theta_1}$
and $\frac{1-\tan\theta_1}{1+\tan\theta_1}$ respectively when $c=c^*$ as we can see in Figure \ref{ex5} (b).

\section{Dual interpretation}
\setcounter{equation}{0}
\setcounter{figure}{0}
\setcounter{table}{0}
\label{sec4}

\subsection{Chambolle and Pock's primal dual algorithm}
\label{sec41}
The algorithm (\ref{alg1}) is equivalent to a special case of Chambolle and
Pock's primal-dual algorithm \cite{Chambolle:2011:FPA:1968993.1969036}. Let
$w^{k+1}=(x^k-y^{k+1})/\gamma$, then (\ref{alg1}) with $F=\partial f$ and $G=\partial g$ is equivalent to
\begin{equation}\label{alg4}
\begin{cases} {w}^{k+1}&=(\mathrm{I}+\frac{1}{\gamma}
\partial f^*)^{-1}(w^k+\frac{1}{\gamma}(2x^k-x^{k-1}))\\
{x}^{k+1}&=(\mathrm{I}+\gamma \partial g)^{-1}(x^k-\gamma
w^{k+1})\end{cases},\end{equation}
where $f^*$ is the conjugate function of $f$. Its resolvent can be evaluated
by the Moreau's identity, \[x=(\mathrm{I}+\gamma
\partial f)^{-1}(x)+\gamma\left(\mathrm{I}+\frac{1}{\gamma}\partial f^*\right)^{-1}\left(\frac{x}{
\gamma}\right).\]

Let $X^n=\frac{1}{n}\sum\limits_{k=1}^n x^k$ and $W^n=\frac{1}{n}\sum\limits_{k=1}^n w^k$, then the
duality gap of the point $(X^n,W^n)$ converges with the rate $\mathcal{O}(\frac
1n)$. See \cite{Chambolle:2011:FPA:1968993.1969036} for the proof. If $f(x)=\|x\|_1$ and 
$g(x)=\iota_{\{x:Ax=b\}}$, then $w^{k}$
will converge to a dual certificate $\eta\in \partial\|x^*\|_1\cap
\mathcal{R}(A^T)$.

\subsection{Alternating direction method of multipliers}
In this subsection we recall the the widely used \textit{alternating direction method of multipliers} (ADMM), which serves as a preliminary for the next subsection.
ADMM \cite{GlowinskiMarroco,gabay76} was shown in \cite{gabay83} to be equivalent to the Douglas-Rachford splitting on the dual problem. To be more specific, consider
\begin{equation}
 \label{primal}\tag{P}
\min_{z\in\mathbbm{R}^{m}} \Psi(z)+\Phi(Dz),
\end{equation}
where $\Psi$ and $\Phi$ are convex functions and $D$ is a $n\times m$ matrix. The dual problem of
the equivalent constrained form $\min \Psi(z)+\Phi(w) \;\mbox{s.t.}\; Dz=w$ is
\begin{equation}
 \label{dual}\tag{D}
\min_{x\in\mathbbm{R}^n} \Psi^*(-D^T x)+\Phi^*(x).
\end{equation}
By applying the Douglas-Rachford splitting (\ref{alg1}) on $F=\partial[\Psi^*\circ(-D^T)]$ and $G=\partial\Phi^*$, 
one recovers the classical ADMM algorithm for (\ref{primal}),
\begin{equation}
 \label{alg6}\tag{ADMM}
\begin{cases} z^{k+1}=\argmin\limits_z {\Psi(z)+\frac{\gamma}{2}\|\frac{1}{\gamma}x^k+Dz-w^k\|^2}\\
 w^{k+1}=\argmin\limits_w {\Phi(w)+\frac{\gamma}{2}\|\frac{1}{\gamma}x^k+Dz^{k+1}-w\|^2}\\
x^{k+1}=x^k+\gamma(Dz^{k+1}-w^{k+1})
\end{cases},
\end{equation}
with the change of variable $y^k=x^k+\gamma w^k$, and $x^k$ unchanged.

After its discovery, ADMM has been regarded as a special augmented Lagrangian method. It turns out
that ADMM can also be interpreted in the context of Bregman iterations. The split Bregman method \cite{Goldstein:2009:SBM:1658384.1658386} for (\ref{primal})
 is exactly the same as (\ref{alg6}), see \cite{Setzer:2009:SBA:1567735.1567776}.
Since we are interested in Douglas-Rachford splitting for the primal formulation of the $\ell^1$ minimization,
the algorithms analyzed in the previous sections are equivalent to ADMM or split Bregman method applied to the dual formulation.

\subsection{Split Bregman method on the dual problem}
\label{seclbsb}
In this subsection we show that the analysis in Section \ref{sec3} can also be applied to the
split Bregman method on the dual formulation \cite{LBSB}.
The dual problem of $\ell^2$ regularized basis pursuit (\ref{BP2}) can be written as
\begin{equation}
\label{BP3}
  \min_z -b^Tz+\frac{\alpha}{2}\|A^Tz-\mathbbm{P}_{[-1,1]^n}(A^T z)\|^2,
\end{equation}
where $z$ denotes the dual variable, see  \cite{Yin:2010:AGL:2078411.2078424}.

By switching the first two lines in (\ref{alg6}), we get a slightly different version of ADMM:
\begin{equation}
 \label{alg7}\tag{ADMM2}
\begin{cases} 
 w^{k+1}=\argmin\limits_w {\Phi(w)+\frac{\gamma}{2}\|\frac{1}{\gamma}x^k+Dz^{k}-w\|^2}\\
z^{k+1}=\argmin\limits_z {\Psi(z)+\frac{\gamma}{2}\|\frac{1}{\gamma}x^k+Dz-w^{k+1}\|^2}\\
x^{k+1}=x^k+\gamma(Dz^{k+1}-w^{k+1})
\end{cases}.
\end{equation}

The well-known equivalence between (\ref{alg6}) and Douglas-Rachford splitting was first explained in \cite{gabay83}. See also \cite{Setzer:2009:SBA:1567735.1567776, Esser09}. For completeness, we discuss the equivalence between (\ref{alg7}) and Douglas-Rachford splitting.
\begin{thm}
\label{thm41}
The iterates in (\ref{alg7}) are equivalent to the Douglas-Rachford splitting 
(\ref{alg1}) on 
 $F=\partial\Phi^*$ and $G=\partial[\Psi^*\circ(-D^T)]$ with
$y^k=x^{k-1}-\gamma w^{k}$.
\end{thm}
\begin{proof}
For any convex function $h$, we have $\lambda\in\partial h(p)\Longleftrightarrow p\in \partial h^*(\lambda)$, which implies
\begin{equation}
 \label{thm-eq1}
\hat p=\argmin\limits_p h(p)+\frac{\gamma}{2}\|D p-q \|^2\Longrightarrow \gamma(D\hat p-q)=J_{\gamma\partial(h^*\circ(-D^T))}(-\gamma q).
\end{equation}
Applying (\ref{thm-eq1}) to the first two lines of (\ref{alg7}), we get
\begin{equation}
 \label{thm-eq2}
 x^k-\gamma w^{k+1}=J_{\gamma F}(x^k+\gamma Dz^{k})-\gamma Dz^{k}.
\end{equation}
\begin{equation}
 \label{thm-eq3}
x^k+\gamma Dz^{k+1}-\gamma w^{k+1} =J_{\gamma G}(x^k-\gamma w^{k+1}).
\end{equation}
Assuming $y^k=x^{k-1}-\gamma w^{k}$, we need to show that the $(k+1)$-th iterate of (\ref{alg7})
satisfies $y^{k+1} = J_{\gamma F}\circ(2J_{\gamma G}-I)y^k+(I-J_{\gamma G})y^k$
and $x^{k+1}=J_{\gamma G}(y^{k+1})$.

Notice that (\ref{thm-eq3}) implies 
\[ J_{\gamma G}(y^{k})=J_{\gamma G}(x^{k-1}-\gamma w^{k})=x^{k-1}+\gamma Dz^{k}-\gamma w^{k}.\]
So we have \[J_{\gamma G}(y^{k})-y^k=x^{k-1}+\gamma Dz^{k}-\gamma w^{k}-(x^{k-1}-\gamma w^{k})=\gamma Dz^{k},\]
and \[2J_{\gamma G}(y^{k})-y^k=x^{k-1}+2\gamma Dz^{k}-\gamma w^{k}=x^{k-1}+\gamma Dz^{k}-\gamma w^{k}+\gamma Dz^{k}=x^k+\gamma Dz^{k}.\]
Thus (\ref{thm-eq2}) becomes 
\[y^{k+1} = J_{\gamma F}\circ(2J_{\gamma G}-I)y^k+(I-J_{\gamma G})y^k.\]
And (\ref{thm-eq3}) is precisely $x^{k+1}=J_{\gamma G}(y^{k+1})$.
\end{proof}

 Applying (\ref{alg7})
on (\ref{BP3}) with $\Psi(z)=-b^Tz$, $\Phi(z)=\frac{\alpha}{2}\|z-\mathbbm{P}_{[-1,1]^n}(z)\|^2$  and $D=A^T$, 
we recover the LB-SB algorithm in \cite{LBSB},
\begin{equation}
 \label{alg8}\tag{LB-SB}
\begin{cases} 
 w^{k+1}=\argmin\limits_w {\frac{\alpha}{2}\|w-\mathbbm{P}_{[-1,1]^n}(w)\|^2+\frac{\gamma}{2}\|\frac{1}{\gamma}x^k+A^Tz^{k}-w\|^2}\\
z^{k+1}=\argmin\limits_z {-b^Tz+\frac{\gamma}{2}\|\frac{1}{\gamma}x^k+A^Tz-w^{k+1}\|^2}\\
x^{k+1}=x^k+\gamma(A^Tz^{k+1}-w^{k+1})
\end{cases}.
\end{equation}

It is straightforward to check that
$\Psi^*\circ(-A)(x)=\iota_{\{x: Ax=b\}}$ and $\Phi^*(x)=\|x\|_1+\frac{1}{2\alpha}\|x\|^2$. By Theorem \ref{thm41}, (\ref{alg8}) is exactly the same as (\ref{alg3}). Therefore, all the results in Section \ref{sec3} hold for (\ref{alg8}).
In particular, the dependence of the eventual linear convergence rate of (\ref{alg8}) on the parameters is governed by (\ref{eq13}) as illustrated in Figure \ref{bestc}.

\begin{rem}
 Let $z^*$ be the minimizer of (\ref{BP3}) then $\alpha S_1(A^T z^*)$ is the solution to (\ref{BP2}), see \cite{Yin:2010:AGL:2078411.2078424}. So 
$t^k=\alpha S_1(A^T z^k)$ can be used as the approximation to $x^*$, the solution to (\ref{BP2}), as suggested in \cite{LBSB}. By Theorem \ref{thm41}, we can see that 
$x^k$ will converge to $x^*$ too. And it is easy to see that $x^k$ satisfies the constraint $A x^k=b$ in (\ref{alg3}).
But $t^k$ does not necessarily lie in the affine set $\{x: Ax=b\}$. 
Thus
$\{t^k\}$ and $\{x^k\}$ are two completely different sequences even though they both can be used in practice.
\end{rem}

\subsection{Practical relevance}
\label{sec44}
To implement the algorithm exactly as presented earlier, the availability of $A^+$ is necessary. Algorithms such as (\ref{alg2}) and (\ref{alg3}), the same as (\ref{alg8}), are not suitable if $(A A^T)^{-1}$ is prohibitive to obtain. On the other hand, there are quite a few important problems for which 
$(A A^T)^{-1}$ is cheap to compute and store in memory. For instance, $A A^T$ may be relatively small and is a well-conditioned matrix in typical compressive sensing problems. Another example is when $A^T$ represents a tight frame transform, for which $A A^T$ is the identity matrix. 

As for the efficiency of  (\ref{alg8}), see \cite{LBSB} for the comparison of (\ref{alg8}) with other state-of-the-art algorithms.

Next, we discuss several examples of (\ref{alg3}), (\ref{alg8}) and (\ref{GDR2}) for the tight frame of discrete curvelets \cite{CDDY06}, in the scope of an application to interpolation of 2D seismic data. In the following examples, let $C$ denote the matrix representing the wrapping version of the two-dimensional fast discrete curvelet transform \cite{CDDY06}, then $C^T$ represents the inverse curvelet transform and $C^T C$ is the identity matrix since the curvelet transform is a tight frame.


\bigskip \noindent{\bf Example 6}
We construct an example with $A=C^T$ to validate formula (\ref{eq13}).
Consider a random sparse vector $x^*$ with length $379831$ and  $93$ nonzero entries, in the curvelet domain which is the range of the
curvelet transform of $512\times512$ images.
The size of the abstract matrix $C^T$ is $262144\times379831$. Notice that, for any $y\in\mathbbm{R}^{512\times512}$, $C y$ is implemented through fast Fourier transform,
thus the explicit matrix representation of $C$ is never used in computation.
Let $b=C^T x^*$ denote the $512\times512$ image  generated by taking 
the inverse transform of $x^*$, see Figure \ref{curvelet-ex1} (a).

Suppose only the data $b$ is given, to recover a sparse curvelet coefficient, we can solve (\ref{BP}) with $A=C^T$ and $x$ being vectors in curvelet domain.

We use both (\ref{alg2}) and (\ref{alg3}) with $\gamma=2$ and $\alpha=25$ to solve (\ref{BP}). 
Since $A$ is a huge implicitly defined matrix, it is not straightforward to compute the angles exactly by SVD as in small matrices examples.
Instead, we obtain approximately the first principal angle $\theta_1=\arccos (0.9459)$ between $\mathcal{N}(A)$ and $\mathcal{N}(B)$ in a more efficient ad hoc way in Appendix \ref{appendix3}.
 Assuming $\cos\theta_1=0.9459$ and $\frac{\alpha}{\alpha+\gamma}=\frac{25}{27}$, if $y^k$ in (\ref{alg3}) converged to a fixed point of the same type (interior or boundary fixed point) as $y^k$ in (\ref{alg2}), the eventual linear rate of (\ref{alg3}) should
be $\sqrt{\frac{\alpha}{\alpha+\gamma}} \cos\theta_1 $ by (\ref{eq13}).
As we can see in Figure \ref{curvelet-ex1} (b), the error curve 
for (\ref{alg3}) matched well with the eventual linear rate $\sqrt{\frac{\alpha}{\alpha+\gamma}} \cos\theta_1$.

 \begin{figure}[htbp]
  \begin{center}
    \subfigure[The data $b=C^T x^*$.]{\includegraphics[scale=0.40]{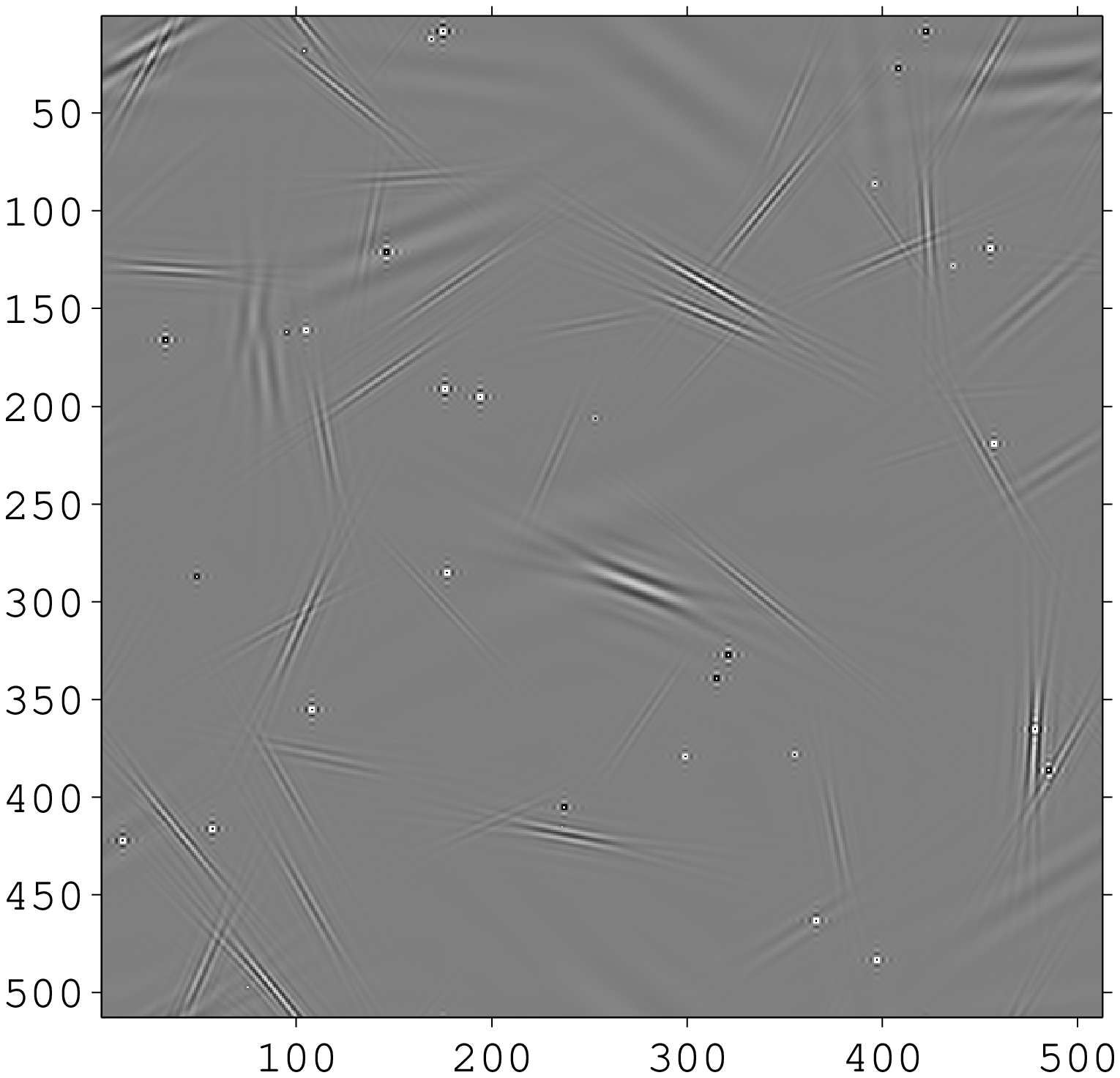}}
     \subfigure[Here $\cos\theta_1=0.9459$. DR stands for (\ref{alg2}) and LBSB stands for (\ref{alg3}) and (\ref{alg8}).  ]{\includegraphics[scale=0.40]{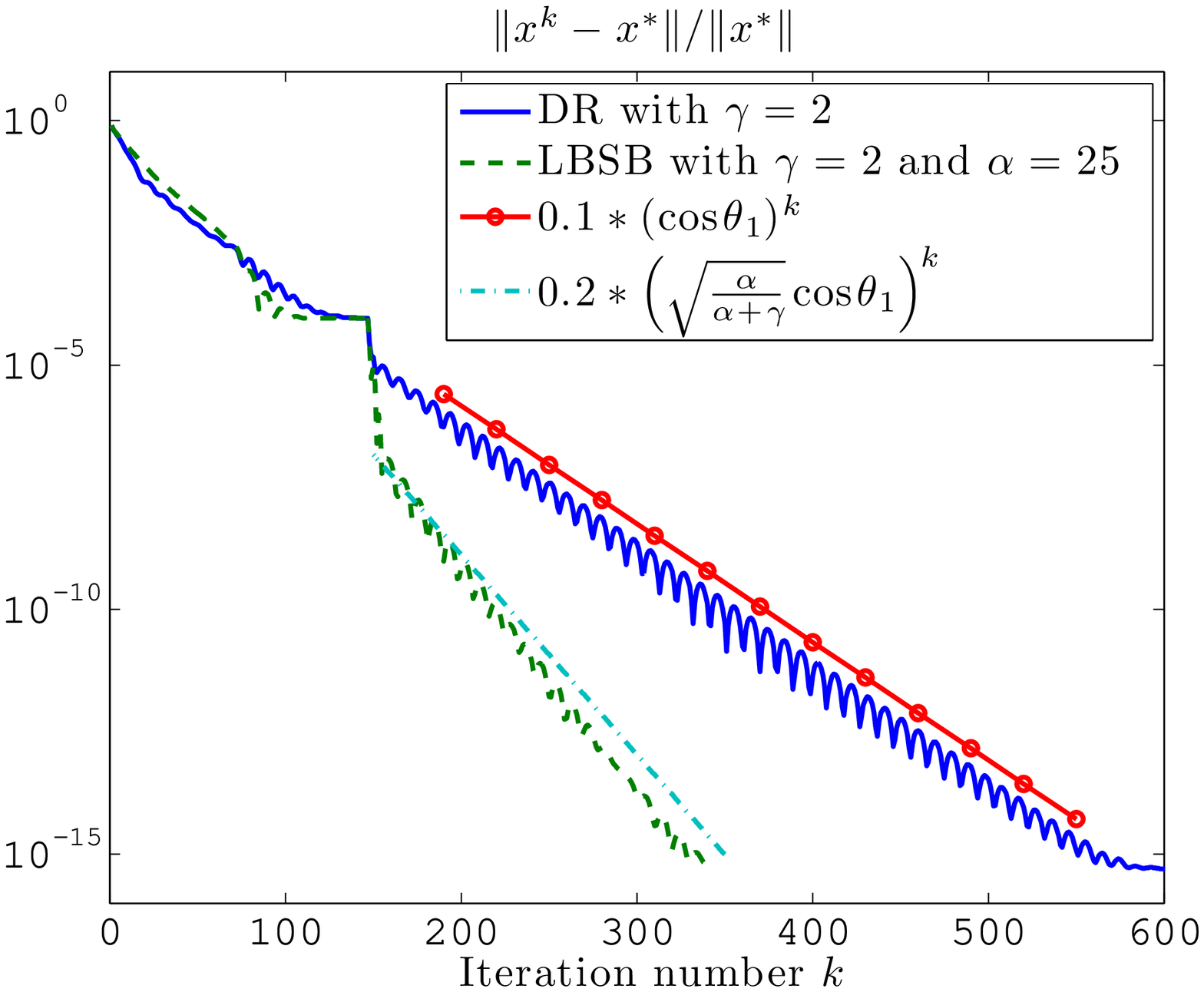}}
  \end{center}
  \caption{Example 6: Recovery of a sparse curvelet expansion.}
  \label{curvelet-ex1}
\end{figure}

 \begin{figure}[htbp]
  \begin{center}
    \subfigure[Left: the original data $b$. Right: reconstructed data with $400$ largest curvelet coefficients $x^*$.]{\includegraphics[scale=0.6]{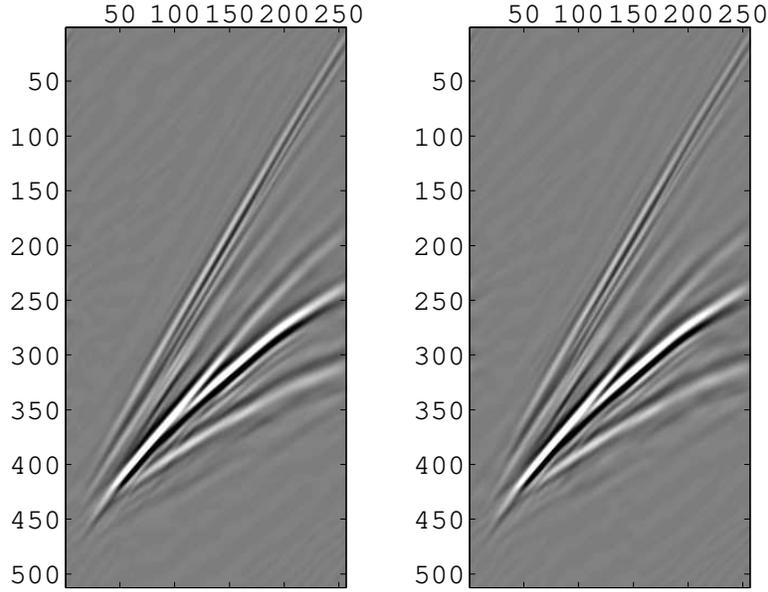}}        
    \subfigure[$\alpha=5$ is fixed. The eventual linear convergence. Douglas-Rachford (LBSB) stands for (\ref{alg3}) and (\ref{alg8}). Peaceman-Rachford
stands for (\ref{GDR2}) with $\lambda=2$.]{\includegraphics[scale=0.6]{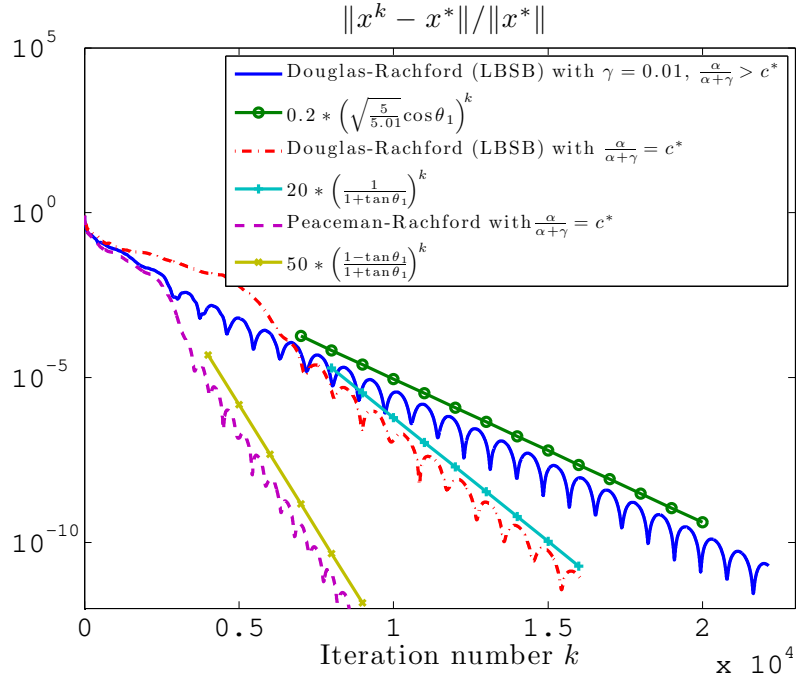}}
  \end{center}
  \caption{Example 7: compression of seismic data.}
  \label{curvelet-ex2}
\end{figure}
 \begin{figure}[htbp]
  \begin{center}
    
    \subfigure[Left: observed data, about $47\%$ random traces missing. Right: recovered data after $200$ iterations with relative error $\|C^Tx^{200}-b\|/\|b\|=2.6\%$ where $b$ is the original data in Figure \ref{curvelet-ex2} (a).]{\includegraphics[scale=0.6]{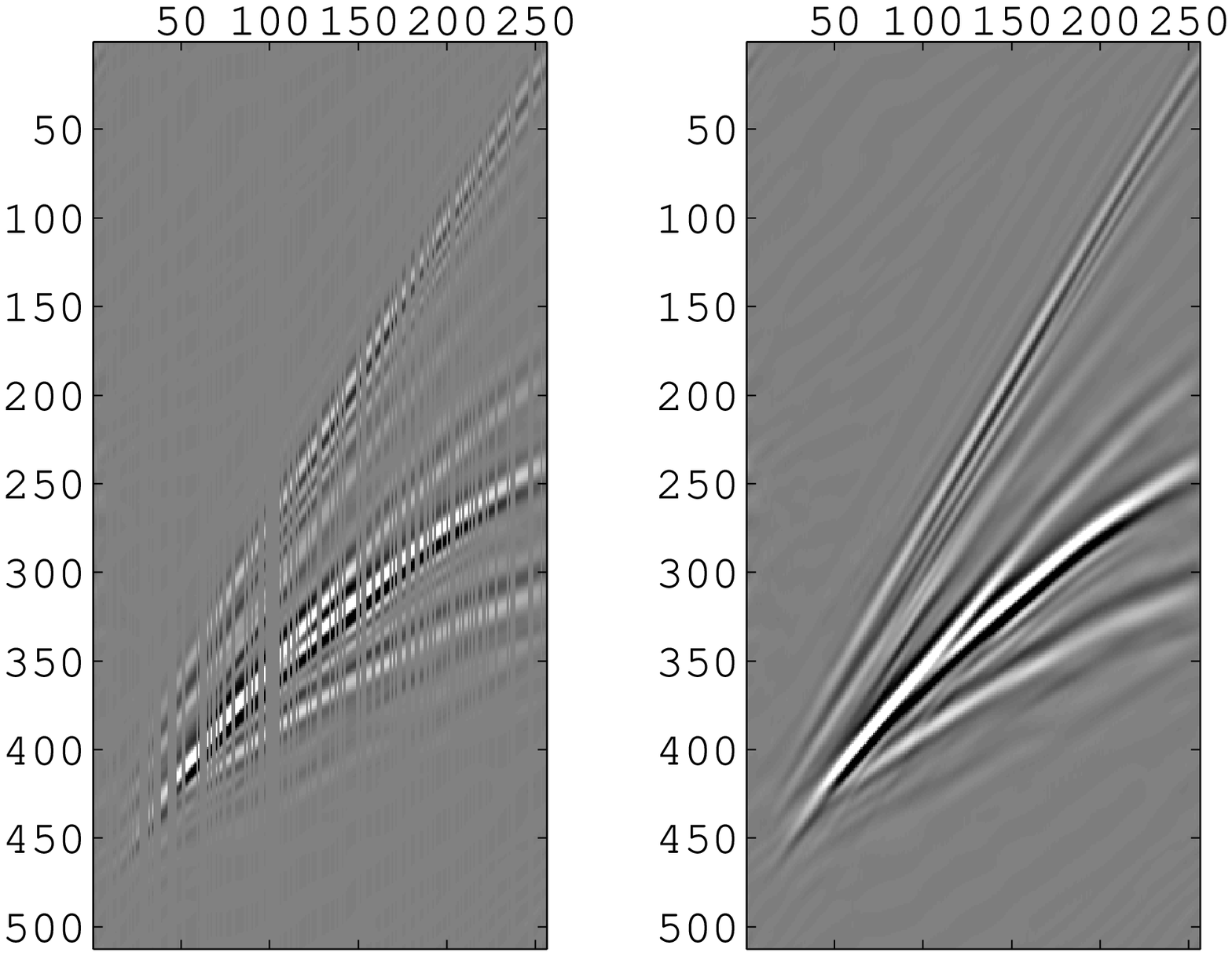}}
\subfigure[Douglas-Rachford (LBSB) stands for (\ref{alg3}) and (\ref{alg8}). Peaceman-Rachford
stands for (\ref{GDR2}) with $\lambda=2$. ]{\includegraphics[scale=0.6]{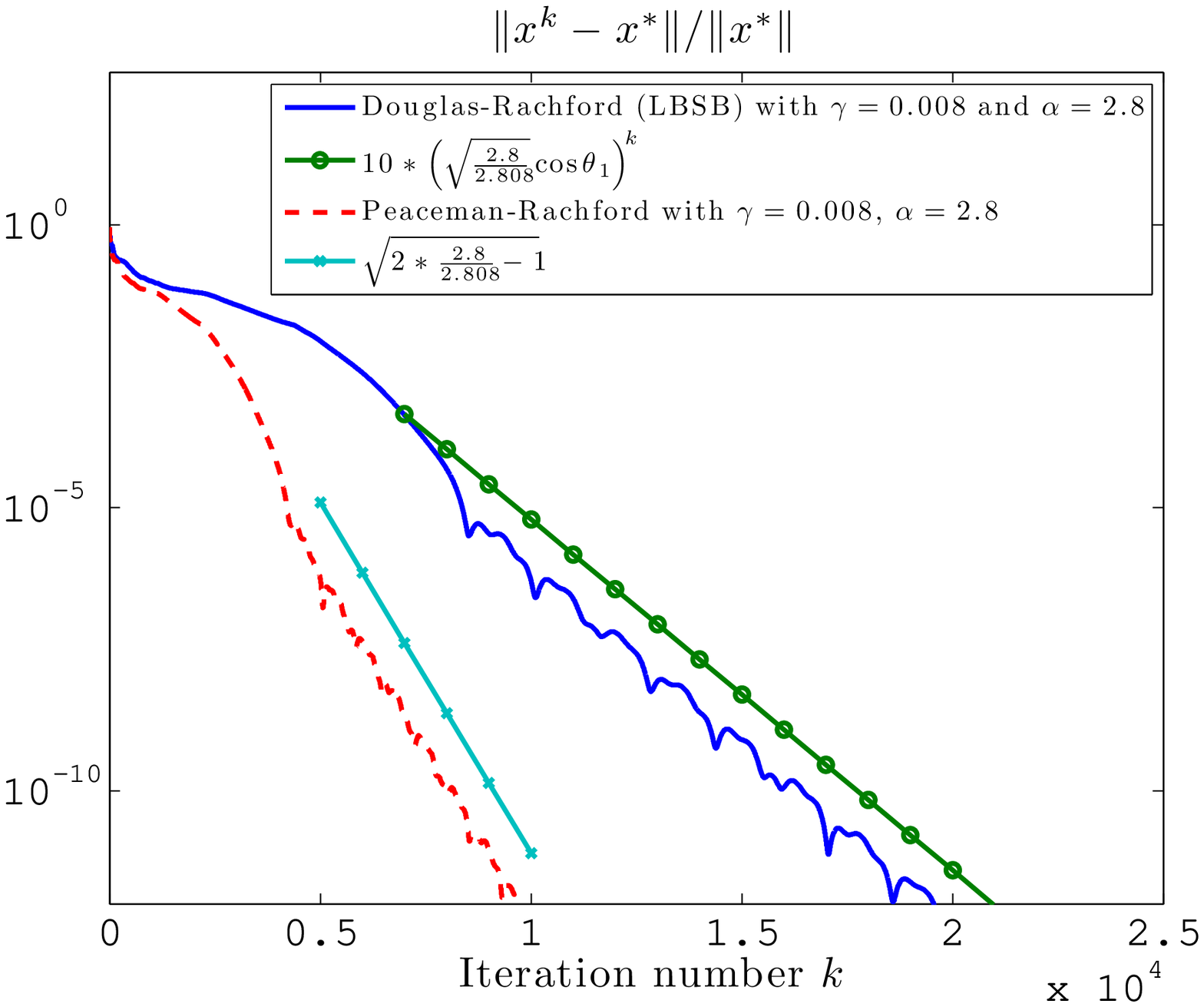}} 
 \end{center}
  \caption{Example 8:  seismic data interpolation.}
  \label{curvelet-ex3}
\end{figure}

\bigskip \noindent{\bf Example 7}
In this example, we consider a more realistic data $b$ as shown in the left panel of Figure \ref{curvelet-ex2} (a).
The data $b$ is generated by the following procedure. First, take a synthetic seismic dataset $\tilde{b}$ consisting of $256$ traces (columns) and $512$ time samples (rows).
Second, solve the basis pursuit $\min\limits_x\|x\|_1$ with $C^Tx=\tilde{b}$ by (\ref{alg3}) up to $50000$ iterations. Third, set the entries in  $x^{50000}$ smaller than $10^{-8}$ to zero and let $x^*$ denote the resulting sparse vector, which has $679$ nonzero entries. Finally, set $b=C^T x^*$.

Given only the data $b$, the direct curvelet transform $C b$ is not as sparse as $x^*$. Thus $C b$ is not the most effective choice to compress the data.
To recover the curvelet coefficient sequence $x^*$, we alternatively solve (\ref{BP}) with $A=C^T$ and $x$ being vectors in curvelet domain.
For this particular example, $x^*$ is recovered. By the method in Appendix \ref{appendix3}, we get $\cos \theta_1=0.99985$. 
To achieve the best asymptotic rate, the parameter ratio $\frac{\alpha}{\alpha+\gamma}$ 
should be $c^*=\frac{1}{(\sin\theta_1+\cos\theta_1)^2}=0.996549$ by (\ref{eq13}). See Figure \ref{curvelet-ex2} (b) for the performance of (\ref{alg8}) and (\ref{GDR2}) with fixed $\alpha=5$ and we can see the asymptotic linear rates match the best rates $\frac{1}{1+\tan\theta_1}$ 
and $\frac{1-\tan\theta_1}{1+\tan\theta_1}$ when $\frac{\alpha}{\alpha+\gamma}=c^*$.

\bigskip \noindent{\bf Example 8} We consider an example of seismic data interpolation via curvelets.
Let $b$ be the same data as in the previous example, see the left panel in Figure \ref{curvelet-ex2} (a).
Let $\Omega$  be the sampling operator corresponding to
 $47$ percent random traces missing, see Figure \ref{curvelet-ex3} (a).

Given the observed data $\bar{b}=\Omega (b)$, to interpolate and recover missing data (traces), one effective model is to pursue sparsity in the curvelet domain \cite{GJI:GJI3698}, i.e., solving
$  \min\limits_x \|x\|_1$ with the constraint 
 $\Omega (C^Tx) =\bar{b}$.   Here $x$ is a vector of curvelet coefficients. If $x^*$ is a minimizer, then $C^T x^*$ can be used as the recovered data. Let $Ax=\Omega (C^Tx)$. Then $A^+=A^T$ since $\Omega$ represents a sampling operator. Thus (\ref{alg3}) and (\ref{alg8}) are straightforward to implement. For this relatively ideal example, the original data $b$ can be recovered. We also observe the eventual linear convergence. 
See Figure \ref{curvelet-ex3} (a) for the recovered data after $200$ iterations of (\ref{alg3}) and (\ref{alg8}).

\section{Conclusion}

In this paper, we analyze the asymptotic convergence rate for Douglas-Rachford
splitting algorithms on the primal formulation
of the basis pursuit, providing a quantification of asymptotic convergence rate of such
algorithms.
In particular, we get 
the asymptotic convergence rates for $\ell^2$-regularized Douglas-Rachford, 
 and the generalized Douglas-Rachford including the Peaceman-Rachford splitting.
The explicit dependence of the convergence rate on the parameters may shed light on how to choose parameters in practice.

\begin{appendices}
  \renewcommand\thetable{\thesection\arabic{table}}
  \renewcommand\thefigure{\thesection\arabic{figure}}
\section{}
\setcounter{equation}{0}

\label{appendix1}
  \begin{lmm}  
 Let $T$ be a firmly non-expansive operator, i.e.,  $\|T(u)-T(v)\|^2\leq \langle u-v,T(u)-T(v)\rangle$ for any $u$
and $v$. Then the iterates $y^{k+1}=T(y^k)$ satisfy $\|y^k-y^{k+1}\|^2\leq \frac{1}{k+1}\|y^{0}-y^*\|^2$ where $y^*$ 
is any fixed point of $T$.
\end{lmm}
\begin{proof}
 The firm non-expansiveness implies 
\begin{eqnarray*}
 \|(I-T)(u)-(I-T)(v)\|^2&=&\|u-v\|^2+\|T(u)-T(v)\|^2-2\langle u-v,T(u)-T(v)\rangle \\
&\leq&\|u-v\|^2-\|T(u)-T(v)\|^2.
\end{eqnarray*}
  Let $u=y^*$ and $v=y^k$,
then
\[\|y^{k+1}-y^k\|^2\leq\|y^{k}-y^*\|^2-\|y^{k+1}-y^*\|^2.\]
Summing the inequality above, we get $\sum\limits_{k=0}^{\infty} \|y^{k+1}-y^k\|^2\leq\|y^{0}-y^*\|^2$. By the firm non-expansiveness and the Cauchy-Schwarz inequality,  we have
$ \|y^{k+1}-y^k\|\leq \|y^{k}-y^{k-1}\|,$
which
implies $\|y^{n+1}-y^n\|^2\leq \frac{1}{n+1}\sum\limits_{k=0}^{n}\|y^{k+1}-y^k\|^2\leq \frac{1}{n+1}\sum\limits_{k=0}^{\infty}\|y^{k+1}-y^k\|^2\leq
\frac{1}{n+1}\|y^{0}-y^*\|^2$.
\end{proof}
For the Douglas-Rachford splitting, see \cite{He12onnon-ergodic} for a different proof for this fact.

\section{}
\setcounter{equation}{0}
\label{appendix3}

Suppose $\mathcal{N}(A)\cap\mathcal{N}(B)=\{\mathbf{0}\}$, we discuss an ad hoc way to find an approximation of the first principal angle $\theta_1$ between $\mathcal{N}(A)$ and $\mathcal{N}(B)$. 
Define the projection operators
$P_{\mathcal{N}(A)}(x)=(I-A^+A)x$ and $P_{\mathcal{N}(B)}(x) = (I-B^+B)x$.
Consider  finding a point in the intersections of two linear subspaces, 
\begin{equation}\label{feasibility}\mbox{find}\quad x\in \mathcal{N}(A)\cap\mathcal{N}(B),\end{equation}
by von Neumann's alternating projection algorithm,
\begin{equation}
 \label{pocs} x^{k+1}=P_{\mathcal{N}(A)} P_{\mathcal{N}(B)} (x^k),
\end{equation}
or the Douglas-Rachford splitting, 
\begin{equation}
 \label{dr_fea} y^{k+1}=\frac12[(2P_{\mathcal{N}(A)}-I) (2P_{\mathcal{N}(B)}-I)+I] (y^k),\quad x^{k+1}=P_{\mathcal{N}(B)}(y^{k+1}).
\end{equation}

For the algorithm (\ref{pocs}), we have the error estimate $\|x^{k}\|=\|(I-A^+A)(I-B^+B)^k x^0\|\leq (\cos\theta_1)^{2k}\|x^0\|$ by (\ref{add_appendixc}).

Assume $y^*$ and $x^*$ are the fixed points of the iteration (\ref{dr_fea}). Let $\mathbf{T}=(I-A^+A)(I-B^+B)+I$. For the algorithm (\ref{dr_fea}), by (\ref{eq6}),
 we have
\[\|x^{k+1}-x^*\|\leq \|y^{k+1}-y^*\|=\|\mathbf{T}(y^k-y^*)\|=\|\mathbf{T}^k(y^0-y^*)\|\leq (\cos\theta_1)^{k}\|y^0-y^*\|.\]
\end{appendices}
Notice that $\mathbf{0}$ is the only solution to (\ref{feasibility}). 
 By fitting lines to $\log(\|x^k\|)$ for large $k$ in (\ref{pocs}) and (\ref{dr_fea}), we get an approximation of $2 \log\cos\theta_1 $ and $\log \cos\theta_1$ respectively. In practice, (\ref{pocs}) is better since the rate is faster and $\|x^k\|$ is monotone in $k$.
This could be an efficient ad hoc way to obtain $\theta_1$ when the matrix $A$ is implicitly defined as in the examples in Section \ref{sec44}.

\bibliographystyle{plain}
\bibliography{basispursuit}

\end{document}